\documentclass[a4paper,oneside]{amsart}
\usepackage{amssymb}
\usepackage{amsthm}
\usepackage{amsmath,amscd}
\usepackage[mathscr]{euscript}
\usepackage[all]{xy}
\usepackage[utf8]{inputenc}
\usepackage{lmodern}
\usepackage[T1]{fontenc}
\usepackage[textwidth=14cm,hcentering]{geometry}
\usepackage[colorlinks=true,linkcolor=red,citecolor=blue]{hyperref}
\usepackage{mathtools}
\usepackage{tikz}
\usetikzlibrary{cd}
\usetikzlibrary{arrows,positioning}

\usepackage{subfiles}

\newtheorem{theo}{Theorem}[section]
\newtheorem{lemm}[theo]{Lemma}
\newtheorem{prop}[theo]{Proposition}
\newtheorem{coro}[theo]{Corollary}

\newtheorem*{theo*}{Theorem}

\theoremstyle{definition}
\newtheorem{defi}[theo]{Definition}
\newtheorem{cons}[theo]{Construction}
\newtheorem{assu}[theo]{Assumption}
\newtheorem{nota}[theo]{Notation}
\newtheorem{example}[theo]{Example}

\newtheorem{rem}[theo]{Remark}

\newtheorem{conj}[theo]{Conjecture}

\numberwithin{equation}{section}

\newcommand{\op}{^{\mathrm{op}}}
\newcommand{\cat}{\mathbf}
\newcommand{\ucat}{\mathrm} 
\newcommand{\on}{\operatorname}
\newcommand{\id}{\mathrm{id}}
\newcommand{\Fun}{\on{Fun}}
\newcommand{\map}{\on{map}}

\newcommand{\ZZ}{\mathbb{Z}} 
\renewcommand{\SS}{\mathbb{S}} 

\newcommand{\Mod}{\cat{Mod}}

\newcommand{\Ab}{\ucat{Ab}}

\newcommand{\THH}{\mathrm{THH}} 

\newcommand{\HH}{\mathrm{HH}} 

\newcommand{\HML}{\mathrm{HML}} 

\newcommand{\Ch}{\ucat{Ch}} 

\newcommand{\Alg}{\cat{Alg}}

\newcommand{\Sp}{\cat{Sp}}

\newcommand{\bQ}{\overline{Q}}

\newcommand{\FF}{\mathbb F}

\title[A multiplicative comparison of HML and THH]{A multiplicative comparison of MacLane homology and topological Hochschild homology}

\author{Geoffroy Horel}
\author{Maxime Ramzi}

\thanks{Geoffroy Horel was partially supported by the project ANR-16-CE40-0003 ChroK funded by the Agence Nationale pour la Recherche. Maxime Ramzi was supported by École Normale Supérieure under the status of ``fonctionnaire stagiaire'', and was hosted by the GeoTop center at K\o benhavns Universitet during the last stages of redaction}

\begin{document}

\address{Université Sorbonne Paris Nord, LAGA, CNRS (UMR 7539), 93430, Villetaneuse, France.}
\address{\'Ecole normale sup\'erieure, DMA, CNRS (UMR 8553), 45 rue d'Ulm, 75230 Paris Cedex 05, France}
\email{horel@math.univ-paris13.fr}
\address{Department of Mathematical Sciences, University of Copenhagen, Universitetsparken 5, DK-2100 Copenhagen, Denmark. }
\email{maxime.ramzi@ens.psl.eu}

\begin{abstract}
Let $Q$ denote MacLane's $Q$-construction, and $\otimes$ denote the smash product of spectra. In this paper we construct an equivalence $Q(R)\simeq \ZZ\otimes R$ in the category of $A_\infty$ ring spectra for any ring $R$, thus proving a conjecture of Fiedorowicz, Pirashvili, Schw\"{a}nzl, Vogt and Waldhausen. More precisely, we construct a symmetric monoidal structure on $Q$ (in the $\infty$-categorical sense) extending the usual monoidal structure, for which we prove an equivalence $Q(-)\simeq \ZZ\otimes -$ as symmetric monoidal functors. From this result, we obtain a new proof of the equivalence $\HML(R,M)\simeq \THH(R,M)$ originally proved by Pirashvili and Waldaushen. This equivalence is in fact made symmetric monoidal, and so it also provides a proof of the equivalence $\HML(R)\simeq \THH(R)$ as $E_\infty$ ring spectra, when $R$ is a commutative ring. 
\end{abstract}

\keywords{}

\maketitle

\section*{Introduction}
In 1992, Pirashvili and Waldhausen \cite{PirWald} proved  that MacLane homology and topological Hochschild homology of a ring, functors introduced earlier by MacLane and Bökstedt respectively, were isomorphic. 

They actually proved it for both homology theories with various coefficients; although their proof worked only at the homology level, and not at the level of the underlying spectra or chain complexes -- indeed it proceeds by showing that both homologies are universal $\delta$-functors for a certain functor. 

In 1995, Fiedorowicz, Pirashvili, Schwänzl, Vogt and Waldhausen \cite{FPSVW} outlined a ``brave new algebra'' proof of the same result, which relied on an analysis of the underlying spectra. The key ingredient in their strategy was the following conjecture.

\begin{conj}
For any discrete ring $R$, there is an equivalence $Q(R)\simeq \ZZ\otimes R$ as $A_\infty$ ring spectra.
\end{conj}

Here $Q$ denotes MacLane's $Q$-construction, which is used to define MacLane homology, and $\otimes$ is our notation for the smash-product of spectra. Unfortunately, the authors of \cite{FPSVW} do not prove this conjecture and in fact, it seems that they were not completely convinced that it was true at the time (see \cite[Remark 3.9]{FPSVW}). From this conjecture, the equivalence between MacLane homology and topological Hochschild homology at the spectrum level follows essentially formally, via general base-change arguments for Hochschild homology. 

The purpose of this paper is to prove that conjecture, and conclude that there is an equivalence between topological Hochschild homology and MacLane homology at the spectrum level. In fact, our proof even provides an equivalence of $E_\infty$ ring spectra when $R$ is assumed to be commutative. The $E_\infty$-structure on $\ZZ\otimes R$ is clear when $R$ is commutative, and on $Q(R)$ it comes from an $E_\infty$-monoidal structure on the functor $Q$ itself, first constructed by Richter \cite{richter} and which we obtain in a different way through the universal property of $Q$. This $E_\infty$-equivalence implies an equivalence of the multiplicative structures on MacLane homology and topological Hochschild homology when $R$ is commutative, which was not known even at the level of homotopy groups, as far as the authors know.

The paper is organized as follows:
\begin{itemize}
    \item In Section \ref{section : HH}, we review Hochschild homology in an $\infty$-categorical setting, in order to establish the desired base-change formula which allows us to go from the conjecture to the result about MacLane homology and topological Hochschild homology.
    \item In Section \ref{section : QC}, we quickly review the definition and basic properties of the $Q$-construction, which we need for later work.
    \item In Section \ref{section : QCGD}, we review a theorem of Johnson and McCarthy \cite{johnsonlinearization} which interprets the $Q$-construction as a first Goodwillie derivative.
    \item  In Section \ref{sec : sym mon structure}, we construct natural symmetric monoidal structures on various categories which allow us to define a symmetric monoidal structure on the $Q$-construction.
    \item In Section \ref{section : sym mon equivalence}, we compare two symmetric monoidal structures in order to finally prove the conjecture in question.
    \item  In Section \ref{section : Richter}, we compare our symmetric monoidal structure on $Q$ to the one constructed by Richter in \cite{richter}.
    \item Section \ref{section : conclusion} is the conclusion, where we deduce the results about MacLane homology and topological Hochschild homology.
    \item In Appendix \ref{appendix : remark}, we explain Remark 3.9 in \cite{FPSVW}.
	\item In Appendix \ref{appendix : monoidal fibration} we prove a technical result that we needed in our proof and which could potentially be of independant interest.
\end{itemize}

\section*{Acknowledgements}
The authors are grateful to Birgit Richter and Teimuraz Pirashvili for a helpful conversation which led to the remark in Appendix A; to Thomas Nikolaus for spotting a mistake in one of the proofs in an earlier version of this work and for noting that a slight modification of the argument in Appendix A could be applied to odd primes as well; to Yonatan Harpaz for suggesting a simpler proof of the main result in Appendix B; and to Brice Le Grignou for suggesting a helpful reference for Section \ref{section : Richter}. We also thank the anonymous referee for making several useful comments about the first draft of this paper.

Finally the second named author would like to thank the first named author for suggesting this subject for his master's thesis.

\section*{Conventions}

By default the categorical terminology (limit, colimit, adjoints, etc.) will refer to the $\infty$-categorical notion. We will sometimes have to work in model categories, in that case we will use the standard terminology (homotopy limits and colimits, left/right Quillen functor, etc.) We write $\infty$-categories in boldface and $1$-categories in regular.

For $\cat{C}$ a category (recall that in our convention this means an $\infty$-category) and $S$ a collection of morphisms in $\cat{C}$, we denote by $S^{-1}\cat{C}$ the localization of $\cat{C}$ at $S$. If $\ucat{M}$ is a model category we denote by $W^{-1}\ucat{M}$ the localization of $\ucat{M}$ at its weak equivalences (the class $W$ will always be clear from context). For instance for $W$ the class of quasi-isomorphisms of chain complexes of abelian groups, we have an equivalence
\[W^{-1}\Ch_*(\ZZ)\simeq \Mod_{\ZZ}\]
where $\Mod_{\ZZ}$ denotes the category of $\ZZ$-modules in spectra.

We use $\otimes$ to denote the smash-product of spectra. We do not make a distinction notationally between a discrete or dg-ring and its associated Eilenberg-MacLane spectrum.
 
\section{Hochschild homology}\label{section : HH}

In this section, we review the definition of Hochschild homology with coefficients in a bimodule in an $\infty$-categorical setting. We denote by $\mathcal{AB}$ the $2$-colored operad whose algebras are pairs $(A,M)$ consisting of an algebra and a bimodule. We denote by $\mathcal{AB}^{\otimes}$ the category over $\ucat{Fin}_*$ constructed from $\mathcal{AB}$ using \cite[Construction 2.1.1.7]{HA}.

We denote by $\mathcal{AB}^{\otimes}_{act}$ the subcategory of $\mathcal{AB}^\otimes$ where we only allow morphisms that send the base point and nothing else to the base point. More concretely, this category can be described as follows.

\begin{cons}
The objects of $\mathcal{AB}^{\otimes}_{act}$ are pairs $(S,U)$ where $S$ is a finite set and $U$ is a subset of $S$. 

A morphism in $\mathcal{AB}^{\otimes}_{act}$ from $(S,U)$ to $(T,V)$ is the data of
\begin{itemize}
\item A map $f:S\to T$ whose restriction to $U$ is a bijection from $U$ to $V$.
\item A total order on the fiber $f^{-1}(t)$ for each $t$ in $T$.
\end{itemize}

Given a morphism $f:(S,U)\to (T,V)$ and $g:(T,V)\to (X,W)$, the composition is constructed as follows.
\begin{itemize}
\item The map of sets $g\circ f:S\to X$ is simply the composition of $g$ and $f$.
\item The total order on $(g\circ f)^{-1}(x)$ is given by the following concatenation of ordered sets
\[(g\circ f)^{-1}(x)=\star_{t\in g^{-1}(x)}f^{-1}(t).\]
\end{itemize}
\end{cons}

Given an $\mathcal{AB}$-algebra in a symmetric monoidal $\infty$-category $\cat{C}$, there is an induced functor
\[\mathcal{AB}^{\otimes}_{act}\to \cat{C}_{act}^{\otimes}\]
that we can compose with the functor $\cat{C}_{act}^{\otimes}\to \cat{C}$ (this is similar to \cite[Definition III.2.3]{nikolaustopological}) to produce a functor $\mathcal{AB}^{\otimes}_{act}\to \cat{C}$. Informally, if we think of an $\mathcal{AB}$-algebra as a pair $(A,M)$ with $A$ an associative algebra and $M$ a bimodule, this functor sends an object $(S,U)$ of $\mathcal{AB}^\otimes_{act}$ to the tensor product $A^{\otimes S-U}\otimes M^{\otimes U}$. A morphism $f:(S,U)\to (T,V)$ induces the map 
\[\overline{f}:A^{\otimes S-U}\otimes M^{\otimes U}\to A^{\otimes T-V}\otimes M^{\otimes V}\]
given by multiplying the copies of $A$ and $M$ using the algebra structure of $A$, the bimodule structure of $M$ and the data of the linear order on the fibers of $f$.

\begin{rem}
If $\cat{C}^{\otimes}$ is a symmetric monoidal $1$-category, the informal description above is in fact an accurate description of this functor.
\end{rem}

\begin{cons}
Now, we construct a functor
\[\chi:\Delta\op\to\mathcal{AB}^{\otimes}_{act}.\]
For this purpose it will be convenient to identify $\Delta\op$ with the category whose objects are totally ordered sets $[n]=\{0,1,\ldots,n+1\}$ and morphisms are order preserving maps that preserve the minimal and maximal element. With this description, the functor $\chi$ can be constructed as follows.
\begin{enumerate}
\item On objects, we set $\chi([n])=(\{0,1,\ldots,n\},0)$.
\item To a map $f:[n]\to[m]$, we assign a map $\chi(f): (\{0,1,\ldots,n\},0)\to (\{0,1,\ldots,m\},0)$ as follows :
\begin{align*}
\chi(f)(i)=f(i)&\;\,\textrm{if}\;\,f(i)\neq m+1\\
\chi(f)(i)=0&\;\,\textrm{if}\;\,f(i)= m+1.\\
\end{align*}
\item For $j\in\{1,\ldots,m\}$, the linear order on $\chi(f)^{-1}(j)=f^{-1}(j)$ is the obvious one. The linear order on $\chi(f)^{-1}(0)$ is the concatenation $ f^{-1}(m+1)\star f^{-1}(0)$
\end{enumerate}
\end{cons}

\begin{defi}
Given an $\mathcal{AB}$-algebra $(A,M)$ in a symmetric monoidal category $\cat C^{\otimes}$, the \emph{cyclic bar construction} of $(A,M)$ denoted $B^{cy}(A,M)$ is the following composite
\[\Delta\op\xrightarrow{\chi}\mathcal{AB}^{\otimes}_{act}\xrightarrow{(A,M)}\cat{C}^{\otimes}_{act}\to \cat{C}.\]
If the colimit of the cyclic bar construction of $(A,M)$ exists in $\cat{C}$ we call it the \emph{Hochschild homology} of $A$ with coefficients in $M$ and denote it by $\HH(A,M)$.
\end{defi}

\begin{rem}
Contrary to what the name suggests, the cyclic bar construction does not extend in general to a cyclic object unless we use $A$ as a bimodule of coefficients.
\end{rem}

\begin{nota}
We follow the standard convention and write $\THH(A,M)$ instead of $\HH(A,M)$ if $\cat{C}$ is the category of spectra.
\end{nota}
\begin{rem}
Assume that $\cat{C}$ is actually a symmetric monoidal $1$-category. In that case, a $\mathcal{AB}$-algebra is simply a pair $(A,M)$ with $A$ an associative algebra in $\cat{C}$ and $M$ a bimodule. The cyclic bar construction that we construct is the usual simplicial object $\Delta\op\to\cat{C}$ sending $[n]$ to $M\otimes A^{\otimes n}$ with all the face maps but the last given by multiplying two adjacent factors and the last face map given by multiplying the factor $M$ with the last copy of $A$ using the left $A$-module structure on $M$ (see for instance \cite[Section 2.1]{FPSVW} for a precise definition of this simplicial object).

Now, assume that our symmetric monoidal category $\cat{C}^{\otimes}$ is obtained by localizing a symmetric monoidal model category $(\ucat{D},\otimes)$ at its weak equivalences. In this situation we have a symmetric monoidal localization functor $\ucat{D}^c\to \cat C$, where $\ucat D^c$ denotes the full subcategory of $\ucat D$ spanned by cofibrant objects \cite[Proposition 4.1.7.4., Example 4.1.7.6.]{HA}. Then we can consider the following commutative diagram
\[
\xymatrix{
\cat{Alg}_{\mathcal{AB}}(\ucat{D}^{c,\otimes})\ar[r]^{B^{cy}}\ar[d]&\Fun(\Delta\op,\ucat{D}^c)\ar[d]\ar[rd]^{\mathrm{hocolim}}&\\
\cat{Alg}_{\mathcal{AB}}(\cat{C}^{\otimes})\ar[r]_{B^{cy}}&\Fun(\Delta\op,\cat{C})\ar[r]_-{\mathrm{colim}}&\cat{C}\\
}
\]
in which the vertical arrows are induced by the symmetric monoidal localization functor $\ucat{D}^{c,\otimes}\to\cat{C}^\otimes$. The commutation of the square is immediate and the commutation of the triangle is the definition of the $\mathrm{hocolim}$ functor. From this diagram, we see that, given an $\mathcal{AB}$-algebra $(A,M)$ in $\ucat{D}^c$, Hochschild homology in our sense of its image in $\cat{Alg}_{\mathcal{AB}}(\cat{C}^\otimes)$ is given by the homotopy colimit of its cyclic bar construction $B^{cy}(A,M)$. This implies that our construction is a generalization of the classical definition of topological Hochschild homology as given for example in \cite[Chapter IX, Definition 2.1]{EKMM} in the case of EKMM spectra or in \cite[Section 4.1]{ShipleySymmetric} in the case of symmetric spectra.
\end{rem}

\begin{prop}
Assume that $\cat{C}^{\otimes}$ is a symmetric monoidal category with sifted colimits and that the symmetric monoidal structure commutes with sifted colimits in each variable. Then the Hochschild homology functor can be promoted to a symmetric monoidal functor
\[\Alg_{\mathcal{AB}}(\cat{C}^{\otimes})\to\cat{C}.\]
\end{prop}

\begin{proof}
Indeed this functor can be written as the following composition of symmetric monoidal functors
\[\cat{Alg}_{\mathcal{AB}}(\cat{C}^{\otimes})\simeq\Fun^{\otimes}(\mathcal{AB}^{\otimes}_{act},\cat{C}^{\otimes})\to \Fun(\mathcal{AB}^{\otimes}_{act},\cat{C})\to\Fun(\Delta\op,\cat{C})\xrightarrow{\on{colim}}\cat{C}.\]
\end{proof}

\subsection{Base change formula for Hochschild homology}\label{subsection : base change formula}

In this subsection, we specialize to categories of left modules $\Mod_R$ for $R$ a commutative ring spectrum. In that case, we write $\HH_R$ for the Hochschild homology $R$-module in order to keep track of the base ring.

Let $\alpha:R\to S$ be a map of commutative rings. In that case, we have the change of scalars adjunction
\[\alpha_!:\Mod_R\leftrightarrows\Mod_S:\alpha^*\]
with $\alpha_!$ symmetric monoidal and $\alpha^*$ lax symmetric monoidal.

Let $A$ be an associative algebra in $\Mod_R$ and $M$ be an $\alpha_!(A)$-bimodule. Since $\alpha^*$ is lax monoidal, $\alpha^*(M)$ inherits an $\alpha^*(\alpha_!A)$-bimodule structure. We can restrict along the map of associative algebras $A\to \alpha^*(\alpha_!A)$ and view the pair $(A,\alpha^*M)$ as an object in $\cat{Alg}_{\mathcal{AB}}(\Mod_R)$. In that case, we can consider the following composition
\begin{equation}\label{eqn : base change simplicial}
\alpha_![B^{cy}(A,\alpha^*M)]\xrightarrow{\simeq} B^{cy}(\alpha_!A,\alpha_!\alpha^*M)\to B^{cy}(\alpha_!A,M)
\end{equation}
where the first map comes from the fact that $\alpha_!$ is symmetric monoidal and the second map comes from the fact that the counit map $\alpha_!\alpha^*M\to M$ is a map of $\alpha_!A$-bimodules. Taking colimits (using the fact that $\alpha_!$ commutes with colimits), we obtain a map
\begin{equation}\label{eqn : base change}
\alpha_!\HH_R(A,\alpha^*M)\to \HH_S(\alpha_!A,M).
\end{equation}

\begin{prop}\label{prop : base change}
The map 
\begin{equation}\label{eqn : base change adjoint}
\HH_R(A,\alpha^*M)\to \alpha^*\HH_S(\alpha_!A,M)
\end{equation}
adjoint to the map \ref{eqn : base change} is an equivalence.
\end{prop}

\begin{proof}
We start with a simple observation. Take $X$ in $\Mod_S$ and $Y$ in $\Mod_R$. Then we can consider the following composite
\[
\alpha_!(\alpha^*X\otimes_R Y)\xrightarrow{\simeq}\alpha_!\alpha^*X\otimes_S\alpha_!Y\to X\otimes_S\alpha_!Y\]
where the second map is induced by the counit of the adjunction. We claim that the adjoint map
\[\alpha^*X\otimes_R Y\to \alpha^*(X\otimes_S\alpha_!Y)\]
is an equivalence. Indeed, this map is natural in $Y$ and both functors of $Y$ preserve colimits, therefore, it suffices to prove it for $Y=R$ which is straightforward.

We can now prove the proposition. The map \ref{eqn : base change adjoint} is obtained by applying the colimit functor to the map
\begin{equation}\label{eqn : base change simplicial adjoint}
B^{cy}(A,\alpha^*M)\to\alpha^* B^{cy}(\alpha_!A,M)
\end{equation}
that is adjoint to the map \ref{eqn : base change simplicial} (this uses the fact that the functor $\alpha^*$ preserves colimits). Therefore, it is enough to prove that the map \ref{eqn : base change simplicial adjoint} is degreewise an equivalence. In simplicial degree $n$, this map is simply the map
\[\alpha^*M\otimes_RA^{\otimes n}\to M\otimes_S\alpha^*(\alpha_!A^{\otimes n})\]
that we considered in the first paragraph of this proof and showed is an equivalence.
\end{proof}
Now we want to make the proposition above functorial and symmetric monoidal in the data. We introduce a category $\cat{Alg}_{\mathcal{AB}}(R,S)$. Informally, this category is the category of pairs $(A,M)$ with $A$ an associative algebra over $R$ and $M$ a bimodule over $\alpha_!(A)$. The precise definition is via the following cartesian square in $\cat{Pr}^L$:
\[
\xymatrix{
\cat{Alg}_{\mathcal{AB}}(R,S)\ar[r]\ar[d]&\cat{Alg}_{\mathcal{AB}}(\Mod_S)\ar[d]^{res}\\
\cat{Alg}_{\mathcal{A}ss}(\Mod_R)\ar[r]_{\alpha_!}&\cat{Alg}_{\mathcal{A}ss}(\Mod_S)
}
\]
Recall that $\cat{Pr}^L$ is the category of presentable categories \cite[Definition 5.5.3.1]{HTT}. 

The functor $res$ is induced by the obvious inclusion of operads $\mathcal{A}ss\to \mathcal{AB}$. Hence, a slightly more formal definition of an object of $\cat{Alg}_{\mathcal{AB}}(R,S)$ is a triple $(A,B,M)$, with $(B,M)$ an $\mathcal{AB}$-algebra in the category of $S$-modules, $A$, an associative algebra in $R$-modules and the data of an equivalence $\alpha_!(A)\simeq B$. Note that the functor $res$ is indeed a left adjoint as shown in the following proposition.

\begin{prop}
For any presentable symmetric monoidal category $\cat{C}$, the restriction functor 
\[res: \Alg_{\mathcal{AB}}(\cat{C})\to\Alg_{\mathcal{A}ss}(\cat{C})\]
has a right adjoint.
\end{prop}

\begin{proof}
The functor $res$ is a presentable fibration in the sense of \cite[Definition 5.5.3.2]{HTT}. The analogous situation with $\mathcal{AB}$ replaced by the operad for algebras and left modules is treated in \cite[Corollary 4.2.3.7]{HA}. The case of algebras and bimodules is completely analogous. In particular, it is both a cartesian and a cocartesian fibration. Using \cite[Proposition 2.4.4.9]{HTT}, we can construct a section $\sigma$ of $res$ such that, for each $A\in\Alg_{\mathcal{A}ss}(\cat{C})$, the object $\sigma(A)$ is a terminal object in $res^{-1}(A)$. Then, \cite[Proposition 5.2.4.3]{HTT} asserts that this section is a right adjoint to $res$.
\end{proof}

Observe moreover, that the functor $\alpha_!$ and $res$ are symmetric monoidal functors. It follows that the category $\Alg_{\mathcal{AB}}(R,S)$ inherits a symmetric monoidal structure. 

We shall denote by $\overline{\alpha}_!$ the top horizontal functor in the above square. Informally, we have $\overline{\alpha}_!(A,M)=(\alpha_!(A),M)$. We can also construct a left adjoint functor
\[\beta:\cat{Alg}_{\mathcal{AB}}(\Mod_R)\to\cat{Alg}_{\mathcal{AB}}(R,S).\]
The functor $\beta$ comes from the following commutative square in $\cat{Pr}^L$, using the universal property of pullbacks
\[
\xymatrix{
\cat{Alg}_{\mathcal{AB}}(\Mod_R)\ar[r]^{\alpha_!}\ar[d]_{res}&\cat{Alg}_{\mathcal{AB}}(\Mod_S)\ar[d]^{res}\\
\cat{Alg}_{\mathcal{A}ss}(\Mod_R)\ar[r]_{\alpha_!}&\cat{Alg}_{\mathcal{A}ss}(\Mod_S)
}
\]
Observe, moreover, that the functor $\beta$ inherits the structure of a symmetric monoidal functor from the symmetric monoidal structure on all the functors in the above square. We denote by $\overline{\alpha}^*$ the right adjoint to $\beta$, this is a lax monoidal functor. Informally, it is given by $\overline{\alpha}^*(A,M)=(A,\alpha^*(M))$. This descriptions follows from the equivalence $\cat{Pr}^L\simeq (\cat{Pr}^R)^{op}$.

Now, we can state precisely, the base change formula for Hochschild homology. Let us denote by $\HH_R$ (resp. $\HH_S$) the Hochschild homology functor in $\Mod_R$ (resp. in $\Mod_S$). We have two lax monoidal functors from $\Alg_{\mathcal{AB}}(R,S)$ to $\Mod_R$. One is given by $\alpha^*\circ \HH_S\circ \overline{\alpha}_!$ and the other by $\HH_R\circ \overline{\alpha}^*$.

\begin{theo}\label{theo : base change}
There is a weak equivalence of lax symmetric monoidal functors 
\[ \HH_R\circ \overline{\alpha}^*\xrightarrow{\sim} \alpha^*\circ \HH_S\circ \overline{\alpha}_!\]
\end{theo}

\begin{proof}
The main difficulty is to construct a comparison map between these two functors. The adjunction $(\alpha_!,\alpha^*)$ induces an adjunction
\[ \alpha_!:\Alg_{\mathcal{AB}}(\Mod_R)\leftrightarrows\Alg_{\mathcal{AB}}(\Mod_R):\alpha^*\]
By definition of $\beta$, there is an equivalence of functors $\alpha_!\xrightarrow{\sim}\overline{\alpha}_!\circ\beta$. We can precompose both sides with $\overline{\alpha}^*$ and we get an equivalence
\[\alpha_!\circ \overline{\alpha}^*\xrightarrow{\sim}\overline{\alpha}_!\circ\beta\circ\overline{\alpha}^*.\]
We can then use the natural transformation $\beta\circ\overline{\alpha}^*\to\id_{\Alg_{\mathcal{AB}}(R,S)}$ given by the counit of the adjunction $(\beta,\overline{\alpha}^*)$. This gives a natural transformation
\[\alpha_!\circ \overline{\alpha}^*\xrightarrow{\sim}\overline{\alpha}_!\circ\beta\circ\overline{\alpha}^*\to \overline{\alpha}_!\]
of functors $\Alg_{\mathcal{AB}}(R,S)\to \Alg_{\mathcal{AB}}(\Mod_S)$. 
Now, we can further compose with the functor $\HH_S:\Alg_{\mathcal{AB}}(\Mod_S)\to \Mod_S$. Using the fact that $\alpha_!$ is symmetric monoidal and commutes with colimits, there is a natural equivalence $\alpha_!\circ \HH_R\simeq \HH_S\circ\alpha_!$. So the above natural transformation induces a natural transformation
\[\alpha_!\circ \HH_R\circ \overline{\alpha}^*\to \HH_S\circ\overline{\alpha}_!\]
Finally, by adjunction this gives a natural transformation
\begin{equation}\label{eqn : base change functorial}
\HH_R\circ \overline{\alpha}^*\to \alpha^*\circ\HH_S\circ\overline{\alpha}_!
\end{equation}
Now, all the functor in sight are at least lax symmetric monoidal functors and the natural transformations that we used : $\alpha_!\xrightarrow{\sim}\overline{\alpha}_!\circ\beta$, $\beta\circ\overline{\alpha}^*\to\id$, $\alpha_!\circ \HH_R\simeq \HH_S\circ\alpha_!$ are all natural transformation of lax symmetric monoidal functors. It follows that the natural transformation (\ref{eqn : base change functorial}) is a natural transformation of lax monoidal functors.

The fact that this natural transformation is a natural equivalence is simply the observation that if we apply it to an object $(A,M)$ of $\Alg_{\mathcal{AB}}(R,S)$, the natural transformation above is simply the map that appears in Proposition \ref{prop : base change}.
\end{proof}

\section{The $Q$-construction}\label{section : QC}

In this section, we review Mac Lane's $Q$-construction. We first briefly  recall its definition and some basic facts which will later help us analyze it and understand its universal property. In Section \ref{section : Richter}, we compare some of our constructions with Richter's work in \cite{richter}, so our definition of the $Q$-construction follows this paper.

\subsection{Generalities}

In what follows, $\ZZ[-]$ will denote the functor $\Ab\to \Ab$ sending an abelian group to the free abelian group on its underlying set. 

Let $C_n=\{0,1\}^n$ denote the $n$-cube, let $0_i,1_i: C_n\to C_{n+1}$ denote the maps that insert a $0$ or a $1$ in the $i$-th position, for $1\leq i \leq n+1$. For an abelian group $A$, we denote by $A^{C_n}$ the abelian group of maps $C_n\to A$. Then have the following natural maps $A^{C_{n+1}}\to A^{C_n}$: 
\[R_i'f(e) = f(0_ie), S_i'f(e) = f(1_ie), P_i' = R_i'+S_i'\]
for any $f\in A^{C_{n+1}}$. 

\begin{defi}
We set $Q_n'(A)= \ZZ[A^{C_n}]$, and $R_i,S_i,P_i :Q_{n+1}(A)\to Q_n(A)$ are defined as $\ZZ[R_i'],\ZZ[S_i'],\ZZ[P_i']$ respectively. We define $\delta : Q_{n+1}'(A)\to Q_n'(A)$ using the formula $\delta=\sum_{i=1}^{n+1}(-1)^i(P_i-R_i-S_i)$
\end{defi}

One can check, using cubical identities, that $(Q_*'(A),\delta)$ is a chain complex. 

\begin{rem}
It is important to observe that $\ZZ[-]$ is not an additive functor. It follows that $\ZZ[0] =~\ZZ[P_i'-~R_i'-~S_i']$ is different from $\ZZ[P_i']-\ZZ[R_i']-\ZZ[S_i'] = P_i-R_i-S_i$. This chain complex can be seen as an attempt to compensate this non-additivity. This interpretation is compatible with $Q$'s universal property which we will explain later. 
\end{rem}

The $Q$-construction is then defined as some quotient of $Q'_*$.

\begin{cons}
Let $N_n(A)$ denote the subgroup of $Q_n'(A) = \ZZ[A^{C_n}]$ generated by the $f:C_n\to A$ such that $f\circ 0_i$ or $f\circ 1_i$ is constant equal to $0$ for some $i$. For $n=0$, $N_0(A)$ is the subgroup spanned by the constant $0$ function $C_0\to A$. The collection of subgroups $N_*(A)$ is stable under $\delta$ and is therefore a subcomplex of $Q_*'(A)$. By definition, MacLane's $Q$-construction is the functor $Q_*(A)=Q_*'(A)/N_*(A)$. It is a functor $\Ab\to \Ch_*(\ZZ)$. We will also view it as a functor $\Ab\to \Mod_\ZZ$ without changing the notation.
\end{cons}

We will the following two results about the $Q$-construction~: 

\begin{lemm}[\cite{jibpi}, 2.6]\label{lemm : Q is summand of Q'}
The functor $Q_n:\Ab\to \Ab$ is a direct summand of the functor $Q_n' = \ZZ[-^{C_n}]$.
\end{lemm}

\begin{rem}
It is not the case that $Q_*$ is a direct summand of $Q'_*$ as functors with values in chain complexes. 
\end{rem}

\begin{lemm}[\cite{johnsonlinearization}, 6.3]
The natural map $Q_*(U)\oplus Q_*(V)\to Q_*(U\oplus V)$ is a quasi-isomorphism for any $U,V\in \Ab$. 
\end{lemm}

\subsection{MacLane homology}

In order to define MacLane homology, it is necessary to know that $Q$ has a lax monoidal structure as a functor $\Ab\to \Ch_*(\ZZ)$, induced by the so-called Dixmier products. We will eventually review and promote to a symmetric monoidal structure (on $Q$ as a functor $\Ab\to \Mod_\ZZ$ this time: there is no $1$-categorical symmetric monoidal structure extending the Dixmier products), and a monoidal augmentation $Q_*\to i$, where $i:\Ab\to \Ch_*(\ZZ)$ is the inclusion in degree $0$. 

This monoidal structure is given by 
\[\ZZ[A^{C_n}]\otimes \ZZ[B^{C_m}] \to \ZZ[(A\times B)^{C_n\times C_m}] \to \ZZ[(A\otimes B)^{C_{n+m}}.\]
One checks that this is compatible with the quotient $Q'_*\to Q_*$ and with the differentials, as well as with the augmentation $\ZZ[A]\to A$.

With this definition, one can see any $A$-bimodule as a $Q_*(A)$-bimodule, and thus define~:

\begin{defi}
Let $A$ be a ring and $M$ an $A$-bimodule. Then \emph{the MacLane homology} of $A$ with coefficients in $M$ is defined to be the Hochschild homology (in $\Mod_\ZZ$) of $Q_*(A)$ with coefficients in $M$~: 
\[\HML(A,M) = \HH_\ZZ(Q_*(A),M).\]
\end{defi}

As a functor, this can be defined as the composite 
\[\Alg_{\mathcal{AB}}(\Ab)\to \Alg_{\mathcal{AB}}(\Ch_*(\ZZ)) \to \Alg_{\mathcal{AB}}(\Mod_\ZZ) \xrightarrow{\HH_\ZZ} \Mod_\ZZ\]
where the first map is $(A,M)\mapsto (Q_*(A),M)$. 

\begin{rem}
Note that $Q_*(R)$ is degreewise flat for any $R$, so the infinity-categorical version of Hochschild homology, where tensor products are all derived, agrees with the usual, underived construction and so this functor agrees with the usual MacLane homology. 
\end{rem}

\subsection{The $Q$-construction as a left Kan extension}

\begin{nota}
Let $F_\ZZ$ denote the full subcategory of $\Ab$ on finitely generated free abelian groups (observe that it is essentially small). Let $\Mod_R^c$ denote the full subcategory of $\Mod_R$ on connective modules. 
\end{nota}

\begin{lemm}\label{nonabder}
The inclusion $F_\ZZ \to \Mod_\ZZ^{c}$ exhibits $\Mod_\ZZ^{c}$ as the nonabelian derived category of $F_\ZZ$ (as defined in \cite[5.5.8]{HTT}).
\end{lemm}

\begin{proof}
The nonabelian derived category of $F_\ZZ$ is, by definition, the category $\Fun^\times(F_\ZZ\op,\cat S)$. The category $F_\ZZ\op$ is an algebraic theory, and so $\Fun^\times(F_\ZZ\op,\ucat{sSet})$, the category of product-preserving functors $F_\ZZ\op\to \ucat{sSet}$, is the category of models of $F_\ZZ\op$ in $\ucat{sSet}$, that is, $\ucat{sAb}$, the category of simplicial abelian groups. The latter is known to be Quillen equivalent (and in fact, equivalent) to $\Ch_{\geq 0}(\ZZ)$. Also, by \cite[5.5.9.3]{HTT}, we have an equivalence $W^{-1}\Fun^\times(F_\ZZ\op,\ucat{sSet})\simeq \Fun^\times(F_\ZZ\op,\cat S)$. It follows that 
\[\Mod_\ZZ^{c}\simeq W^{-1}\Ch_{\geq 0}(\ZZ)\simeq W^{-1}\Fun^\times(F_\ZZ\op,\ucat{sSet})\simeq \Fun^\times(F_\ZZ\op,\cat S).\]
Moreover one can easily trace through these equivalences to check that the Yoneda embedding $F_\ZZ\to \Fun(F_\ZZ\op,\cat S)$ is identified with the usual inclusion $F_\ZZ\to \Mod_\ZZ^c$ 
\end{proof}

\begin{rem}\label{remkbadz}
The fact that the model category $\Fun^\times(F_\ZZ\op,\ucat{sSet})$ presents $\Fun^\times(F_\ZZ\op,\cat S)$ is specific to $\ucat{sSet}$ and $\cat S$ and is not true for more general model categories. In fact an easy consequence of some of the results of this paper is that in general it does not hold for $\Ch_*(R)$ and $\Mod_R$. 

These ideas can be traced back to \cite{Badzioch}, where a similar result is proved for an arbitrary algebraic theory. 
\end{rem}

\begin{coro}\label{univnonabder}
For any presentable category $\cat{C}$, the restriction functor
\[\Fun^{sif}(\Mod_\ZZ^c,\cat{C})\to\Fun(F_\ZZ,\cat{C})\]
is an equivalence, where $\Fun^{sif}$ denotes the full subcategory of $\Fun$ on sifted-colimit preserving functors. This equivalence further restricts to an equivalence 
\[\Fun^L(\Mod_\ZZ^c,\cat{C})\to\Fun^\amalg(F_\ZZ,\cat{C})\]
between colimit-preserving functors and coproduct-preserving functors.
\end{coro}

We can now define a new functor denoted $\bQ$.

\begin{defi}\label{defiQ}
We denote by $\bQ$ the functor $\Mod_{\ZZ}^c\to \Mod_{\ZZ}$ corresponding to $Q_{|F_{\ZZ}}$ via the above equivalence.
\end{defi}

\begin{prop}
The restriction of $\bQ$ to $\Ab\subset \Mod_{\ZZ}^c$ is naturally equivalent to $Q$.
\end{prop}

\begin{proof}
The functor $\bQ$ is the left Kan extension of $Q$ along the fully-faithful inclusion $F_\ZZ\to\Mod_\ZZ^c$. This inclusion factors as $F_\ZZ\to \Ab\to \Mod_\ZZ^c$, both of which are fully-faithful, so that the restriction of $\overline Q$ to $\Ab$ is the left Kan extension of $Q$ to $\Ab$. The result then follows from the following lemma.
\end{proof} 

\begin{lemm}
The functor  $Q: \Ab\to \Mod_\ZZ$ is the left Kan extension of its restriction to $F_\ZZ$.
\end{lemm}

\begin{proof}
The Kan extension will be pointwise, so it suffices to prove the result for 
\[Q:~\Ab_\kappa\to~\Mod_\ZZ\]
where $\Ab_\kappa$ is the (essentially small) category of abelian groups of cardinality $<\kappa$ for $\kappa$ an infinite cardinal. 

We restrict to $\Ab_\kappa$ as a technical convenience, to be able to identify $\Fun(\Ab_\kappa,\Mod_\ZZ)$ with $W^{-1}\Fun(\Ab_\kappa,\Ch_*(\ZZ))$, which requires the domain to be (essentially) small. The argument does not depend on it significantly, apart from this. 

The restriction functor $i^* : \Fun(\Ab_\kappa,\Ch_*(\ZZ))\to \Fun(F_\ZZ,\Ch_*(\ZZ))$ is exact, therefore its derived functor $\Fun(\Ab_\kappa, \Mod_\ZZ)\to \Fun(F_\ZZ,\Mod_\ZZ)$ is just restriction, denoted $i^*$ as well. 

Its left adjoint, the $1$-categorical left Kan extension, can be derived. The resulting functor is the $\infty$-categorical left Kan extension. Since $Q: F_\ZZ \to\Ch_*(\ZZ)$ is a connective chain complex of projective functors, it is cofibrant in the projective model structure, and therefore the $\infty$-categorical left Kan extension of $Q$ is simply computed as the $1$-categorical left Kan extension.  

The result now follows from the following lemma.
\end{proof}

\begin{lemm}
Let $\kappa$ be an arbitrary infinite cardinal. Then, as a functor of $1$-categories, $\Ab_\kappa\to \Ch_*(\ZZ)$, the functor $Q$ is the left Kan extension of its restriction to $F_\ZZ$.
\end{lemm}

\begin{proof}
The category $\Ch_*(\ZZ)$ is cocomplete and colimits of chain complexes are computed degreewise. So we only need to prove the result for each $Q_n: \Ab_\kappa\to \Ab$. 

Moreover, note that $F_\ZZ\subset \Ab_\kappa\subset \Ab$ are full-subcategory inclusions, therefore to prove the result about $Q_n : \Ab_\kappa\to \Ab$, it suffices to prove that $Q_n : \Ab\to \Ab$ is the left Kan extension of its restriction to $F_\ZZ$, and the result for $\Ab_\kappa$ will follow. 

Recall from Lemma \ref{lemm : Q is summand of Q'} that $Q_n$ is a direct summand of $Q_n' = \ZZ [(-)^{C_n}]$, and the Kan extension functor is left adjoint so it preserves direct sums; in particular it suffices to prove the result for $Q_n'$. Now the result will follow from the fact that $Q_n'$ preserves sifted colimits, as a composite of two functors that do. 

Given any abelian group $A$, we can form the slice category $D:=F_\ZZ/A$ of pairs 
\[(P,f), P\in F_\ZZ, f:P\to A.\]
We claim that $D$ is sifted. To show this, consider $((P,f),(Q,g))\in D\times D$, note that $(P\oplus Q, f\oplus g)$ is initial in $((P,f),(Q,g))/D$, so that the latter is a weakly contractible category. It follows that $D\to D\times D$ is cofinal, so $D$ is sifted. 

In particular, $Q_n'(A) \cong \mathrm{colim}_{(P,f)\in F_\ZZ/A}Q_n'(P) = \mathrm{Lan}_i (Q_n'\circ i)(A)$, where $i: F_\ZZ\to \Ab$ is the inclusion, so $Q_n'$ is indeed the Kan extension of its restriction. The result follows. 
\end{proof}

\section{The $Q$-construction as a first Goodwillie derivative}\label{section : QCGD}
In this section, we review a theorem of Johnson and McCarthy by interpreting \cite{johnsonlinearization} in terms of Goodwillie derivative. 

We denote by $\Mod_{\ZZ}^c$ the category of connective $\ZZ$-modules. We denote by $\Mod_\ZZ^{c,\omega}$ the category of $\ZZ$-modules that are connective and compact. For $\cat{C}$ a pointed presentable $\infty$-category, we denote by $\Fun^{exc}(\Mod_\ZZ^{c,\omega},\cat{C})$ the full subcategory of $\Fun(\Mod_\ZZ^{c,\omega},\cat{C})$ spanned by functors that are excisive. We will slightly differ from the usual conventions, and say that a functor is excisive if it preserves the terminal object and sends cocartesian squares to cartesian squares (this is usually called \textit{reduced excisive}).

The inclusion into pointed functors (those that preserve the terminal object)
\[\Fun^{exc}(\Mod_\ZZ^{c,\omega},\cat{C})\to \Fun^*(\Mod_\ZZ^{c,\omega},\cat{C})\]
has a left adjoint denoted $D_1$. An explicit construction is given by the following formula
\[D_1F(X)=\on{colim}_n\Omega^nF(\Sigma^n X).\]
This formula is due to Goodwillie in the context of functors from spaces to spaces (see \cite[Section 1]{GoodwillieCalculusIII}). See also \cite[Example 5.3]{KuhnGoodwillie} for more general model categories  and \cite[Example 6.1.1.28]{HA} for an $\infty$-categorical version of this statement.

Recall that the functor $\ZZ[-]:\Ab\to \Ab$ is the functor that sends an abelian group $A$ to the free abelian group generated by the set $A$. We can extend $\ZZ[-]$ to a functor from simplicial abelian groups to simplicial abelian groups. The resulting functor preserves weak equivalences and thus induces a functor $\Mod_{\ZZ}\to\Mod_{\ZZ}$. This functor can easily be seen to be equivalent to the functor $A\mapsto \ZZ\otimes \Sigma^{\infty}_+A$ (recall that $\otimes $ is the tensor product over $\SS$, i.e. the smash product of spectra)

Similarly, there is a reduced version $\tilde{\ZZ}[-]$ that sends $A$ to the quotient $\ZZ[A]/\ZZ[0]$. Applying this functor levelwise we obtain a functor
\[\tilde{\ZZ}[-]:s\Ab\to s\Ab.\]
This functor preserves weak equivalences and thus induces a functor
\[\Mod_{\ZZ}^c\to \Mod_{\ZZ}^{c}.\] 
This last functor can easily be checked to be the functor $A\mapsto \ZZ\otimes\Sigma^\infty A$ (where $A$ is based at $0$).

The main theorem of Johnson and McCarthy is the following.

\begin{theo}\label{QisD1}
There is an equivalence 
\[\overline Q\simeq D_1(\tilde{\ZZ}[-])\]
in the category $\Fun^{exc}(\Mod_{\ZZ}^{c,\omega},\Mod_{\ZZ})$.
\end{theo}

In order to prove Johnson and McCarthy's theorem above, we need to relate $\overline Q$ and the ``naive'' extension of $Q$ to $\Mod_\ZZ^c$ that is defined in \cite{johnsonlinearization}. 

\begin{defi}
We let $\mathbf Q$ denote the naive extension of $Q$ to $\Ch_{\geq 0}(\ZZ)$. Explicitly, this is given by the following composite
\[\Ch_{\geq 0}(\ZZ)\xrightarrow{\on{DK}^{-1}}\mathrm s\Ab\xrightarrow{Q}\mathrm s\Ch_{\geq 0}(\ZZ)\xrightarrow{\on{DK}}\Ch_{\geq 0}\Ch_{\geq 0}(\ZZ)\xrightarrow{\on{Tot}}\Ch_{\geq 0}(\ZZ)\]
where $\on{DK}$ denotes the Dold Kan equivalence and $\on{Tot}$ denotes the total complex of a double complex.
\end{defi}

\begin{prop}
The functor $\mathbf{Q}$ from $\Ch_{\geq 0}(\ZZ)$ to itself preserves quasi-isomorphisms, and filtered colimits. 
\end{prop}

\begin{proof}
The claim for filtered colimits follows from the fact that $Q$ preserves them. This is the case because, by Lemma \ref{lemm : Q is summand of Q'}, the functor $Q$ is degreewise a direct summand of $\ZZ[(-)^{C_n}]$, which clearly preserves filtered colimits.

As for quasi-isomorphisms, this follows from \cite[Theorem 7.5]{johnsonlinearization}, and the fact that $D_1\tilde\ZZ[-]$ preserves quasi-isomorphisms. For the latter, this is the case because it is a filtered colimit of quasi-isomorphism-preserving functors. 
\end{proof}
In particular, $\mathbf Q$ descends to a functor $\Mod_\ZZ^c\to \Mod_\ZZ$. We use the same notation for this new functor. 

\begin{nota}
We denote by $\Fun^{rex}(\cat C,\cat D)$ the full subcategory of $\Fun(\cat C,\cat D)$ on the right exact functors, that is, those that preserve finite colimits. Observe that if $\cat{D}$ is stable and $\cat{C}$ is pointed, we have an equivalence
\[\Fun^{rex}(\cat C,\cat D)\simeq \Fun^{exc}(\cat C,\cat D).\]
\end{nota}

\begin{theo}
There is an equivalence 
\[\mathbf Q\simeq D_1(\tilde{\ZZ}[-]).\]
in the category $\Fun^{rex}(\Mod_{\ZZ}^{c,\omega},\Mod_{\ZZ})\simeq \Fun^{exc}(\Mod_{\ZZ}^{c,\omega},\Mod_{\ZZ})$.
\end{theo}

\begin{proof}
This is \cite[Theorem 7.5]{johnsonlinearization}.
\end{proof}

We can now prove Theorem \ref{QisD1}

\begin{proof}[Proof of Theorem \ref{QisD1}]
Because of the previous theorem, it suffices to prove that $\mathbf Q\simeq \overline Q$ as functors on $\Mod_\ZZ^c$. 
The functor $Q$ preserves direct sums as a functor $F_\ZZ\to \Mod_\ZZ$. It follows that $\overline Q$ lives in $\Fun^L(\Mod_\ZZ^c,\Mod_\ZZ)$. The theorem and proposition above imply that $\mathbf Q$ preserves filtered colimits, and that its restriction to $\Mod_\ZZ^{c,\omega}$ preserves finite colimits, hence $\mathbf Q $ lives in $\Fun^L(\Mod_\ZZ^c,\Mod_\ZZ)$ as well. Therefore, by Corollary \ref{univnonabder} they are equivalent if and only if their restriction to $F_\ZZ$ are. But both these restrictions are just $Q$, by definition. 
\end{proof}

There is also a universal property of the restriction of $Q$ to $F_{\ZZ}$. Let $\Fun^{\oplus}(F_{\ZZ},\Mod_{\ZZ})$ be the category of functors that preserve finite direct sums. Let
\[\on{Add}:\Fun(F_{\ZZ},\Mod_{\ZZ})\to\Fun^{\oplus}(F_{\ZZ},\Mod_{\ZZ})\]
be the left adjoint to the inclusion (its existence follows from Theorem \ref{presentability} in the next section, the proof of which does not depend on the results in this section). Equipped with this functor, we have the following proposition. We were not able to find it in the literature but as is clear from the proof, it is an easy consequence of \cite[Theorem 7.5]{johnsonlinearization}.

\begin{prop}\label{QisAdd}
There is an equivalence
\[Q\simeq\on{Add}(\ZZ[-]).\]
in the category of functors $\Fun(F_\ZZ,\Mod_{\ZZ})$.
\end{prop}

\begin{proof}
Let us consider the following commutative diagram of left adjoints
\[
\xymatrix{
\Fun(F_{\ZZ},\Mod_{\ZZ})\ar[rr]^{\on{Add}}\ar[d]&&\Fun^\oplus(F_{\ZZ},\Mod_{\ZZ})\ar[d]^{\simeq}\\
\Fun(\Mod_{\ZZ}^{c,\omega},\Mod_{\ZZ})\ar[r]_{P}&\Fun^*(\Mod_{\ZZ}^{c,\omega},\Mod_{\ZZ})\ar[r]_{D_1}&\Fun^{exc}(\Mod_{\ZZ}^{c,\omega},\Mod_{\ZZ})
}
\]
The vertical functors are given by left Kan extension and the horizontal functors are the left adjoint to the inclusions. The diagram commutes since the corresponding diagram of right adjoints obviously commutes. Let us start with $\ZZ[-]$ in the top left corner. The left Kan extension of $\ZZ[-]$ to $\Mod_{\ZZ}^{c,\omega}$ is again the functor $\ZZ[-]$. If we apply $P$ to this we obtain the functor $\tilde{\ZZ}[-]$ and then if we apply $D_1$ we obtain $Q$ by the theorem of Johnson and McCarthy above. By the commutativity of the diagram, we obtain the desired result.
\end{proof}
Let $\Ab^{\omega}$ denote the full subcategory of $\Ab$ on finitely generated modules. Then we can deduce from the above proposition a similar result for the restriction of $Q$ to $\Ab^{\omega}$ and $\Ab_\kappa$. 

\begin{coro}\label{QisAdd2}
There is an equivalence $Q\simeq \on{Add}(\ZZ[-])$ in the category  $\Fun(\Ab^{\omega},\Mod_\ZZ)$, as well as in $\Fun(\Ab_\kappa,\Mod_\ZZ)$. 
\end{coro}
\begin{proof}
Consider the following commutative diagram of left adjoints~: 
\[
\xymatrix{\Fun(F_\ZZ,\Mod_\ZZ) \ar[r] \ar[d] & \Fun(\Ab^{\omega},\Mod_\ZZ) \ar[d]\\
 \Fun^\oplus(F_\ZZ,\Mod_\ZZ) \ar[r] & \Fun^\oplus(\Ab^{\omega},\Mod_\ZZ)}
 \]
The horizontal arrows are left adjoint to restriction, and the vertical arrows are left adjoint to inclusion. Since the diagram of right adjoints obviously commutes, this diagram does too. 

Note that $\ZZ[-]: \Ab^{\omega}\to \Mod_\ZZ$ is the left Kan extension of its restriction to $F_\ZZ$ (the proof is the same as for $Q$), so we can compute its image in $\Fun^\oplus(\Ab^{\omega},\Mod_\ZZ)$ by computing the image of $\ZZ[-]$ in the top left hand corner. But this factors through $\Fun^\oplus(F_\ZZ, \Mod_\ZZ)$ so by the previous proposition, this image is the same as the image of $Q$. 

But $Q_{\mid \Ab^{\omega}}$ is also the left Kan extension of its restriction to $F_\ZZ$, so if we start from $Q$ in the top left hand corner, we get $Q$ in $\Fun(\Ab^{\omega},\Mod_\ZZ)$, which is already additive, so we get $Q$ in the bottom right hand corner, which is what we wanted. 

The same proof works verbatim for $\Ab_\kappa$. 
\end{proof}
All the reasonable candidates for $Q$ having been identified, we can now use a common notation for all of them.

\begin{nota}
We let $Q$ denote any of the functors $Q,\overline Q,\mathbf Q, Q_{\mid F_\ZZ}$.
\end{nota}

Here is a summary of this work on the $Q$-construction~: 

There is a functor $Q: \Mod_\ZZ^c\to \Mod_\ZZ$ whose restriction to $\Ab$ is the classical $Q$-construction, which also agrees, on the level of $\Ch_{\geq 0}(\ZZ)$, with the naive extension of $Q$ using the Dold-Kan equivalence. This functor is the left Kan extension of its restriction to $F_\ZZ$ and therefore to all intermediate full subcategories ($\Ab,\Ab^\omega,\Ab_\kappa,\Mod_\ZZ^{c,\omega}$). 

This functor is the ``rexification'' of the functor
\[A\mapsto \ZZ\otimes \Sigma^\infty A\]
and when restricted to $F_\ZZ,\Ab^\omega, \Ab_\kappa$, it is the additivization of $\ZZ[-]$ (or $\tilde\ZZ[-]$).

\section{The symmetric monoidal structure}\label{sec : sym mon structure}

In the following, we denote by $\mathcal{A}$ a small symmetric monoidal category. We give ourselves a set of diagrams
\[f_u:K_{u}^{\triangleleft}\to \mathcal{A}\]
indexed by a set $U$ where the categories $K_u$ are finite. For $\cat{C}$ a category with limits, we denote by $\Fun^U(\mathcal{A},\cat{C})$ the category of functors $F:\mathcal{A}\to\cat{C}$ sending the diagrams in $U$ to limit diagrams.

\begin{example}\label{examples}
\begin{enumerate}
\item If $U=\{u\}$ is a singleton, $K_u=\varnothing$ and $f_u:[0]=\varnothing^{\triangleleft}\to \mathcal{A}$ is the map hitting the object $a\in\mathcal{A}$, then $\Fun^U(\mathcal{A},\cat{C})$ is the category of functors sending $a$ to the terminal object of $\cat{C}$.
\item Assume that $\mathcal{A}$ has finite products, we can take $U$ to be the set of pairs of objects of $\mathcal{A}$ and for $(a,a')\in U$ we consider the diagram
\[a\leftarrow a\times a'\rightarrow a'\]
in $\mathcal{A}$. Then $\Fun^U(\mathcal{A},\cat{C})$ is the category of functors that preserve binary products. If we add the diagram $*= \emptyset^\triangleleft$ with value the terminal object of $\mathcal A$, we get the category of functors preserving arbitrary finite products.
\item Take $U$ to be the set of cartesian squares in $\mathcal{A}$. We can view a cartesian square as a diagram $K^{\triangleleft}\to\mathcal{A}$ with $K$ the category freely generated by the graph $\bullet\rightarrow\bullet\leftarrow\bullet$. In that case, the category $\Fun^U(\mathcal{A},\cat{C})$ is the category of functors that preserve cartesian squares.
\item Using cocartesian squares instead of cartesian squares in the previous example, the category $\Fun^U(\mathcal{A},\cat{C})$ is the category of functors sending cocartesian squares to cartesian squares. 
\item Combining the first and fourth example, we see that the category of excisive functors is a category of the form $\Fun^U(\mathcal{A},\cat{C})$ for a certain choice of $U$.
\end{enumerate}

\end{example}

\begin{theo}\label{presentability}
Let $\cat{C}$ be a presentable category. 
\begin{enumerate}
\item The category $\Fun^U(\mathcal{A},\cat{S})$ is presentable.
\item The canonical functor
\[\Fun^U(\mathcal{A},\cat{S})\otimes\cat{C}\to\Fun^U(\mathcal{A},\cat{C})\]
is an equivalence. In particular, the category $\Fun^U(\mathcal{A},\cat{C})$ is presentable.

\item The inclusion
\[\Fun^U(\mathcal{A},\cat{C})\to \Fun(\mathcal{A},\cat{C})\]
has a left adjoint. 
\end{enumerate}
\end{theo}

\begin{proof}
The first part comes from the observation that $\Fun^U(\mathcal{A},\cat{S})$ is the category of local objects with respect to the maps
\[\on{colim}_{i\in K_u} \map_{\mathcal{A}}({f_u(i)},-)\to \map_{\mathcal{A}}(f_u(\infty_u),-),\]
where $\infty_u$ denotes the tip of the cone $K_u^{\triangleleft}$. This immediatly gives the presentability and the fact that the inclusion $\Fun^U(\mathcal A,\cat S)\to \Fun(\mathcal A,\cat S)$ has a left adjoint, which is given by the localization.

In order to prove the second part, we will use the equivalence  $\cat E\otimes\cat C \simeq \Fun^R(\cat C\op,\cat E)$, for presentable categories $\cat{E,C}$ (see \cite[4.8.1.17]{HA}). Let us consider the full inclusion
\[
\iota : \Fun^R(\cat{C}\op, \Fun^U(\mathcal A,\cat S))\subset  \Fun^R(\cat C\op, \Fun(\mathcal A,\cat S)).\]
The essential image is the category of functors $\cat{C} \op\to \Fun(\mathcal A,\cat S)$ with the property that for each object $c\in\cat{C}\op$ the value of the functor lands in $\Fun^U(\mathcal A,\cat S)$. On the other hand, there is an equivalence
\[\Fun^R(\cat C\op, \Fun(\mathcal A,\cat S))\simeq \Fun(\mathcal{A}, \Fun^R(\cat C\op,\cat S)).\]
Seen through this equivalence, the essential image of the functor $\iota$ above is identified with the full subcategory
\[\Fun^U(\mathcal{A}, \Fun^R(\cat C\op,\cat S))\subset \Fun(\mathcal{A}, \Fun^R(\cat C\op,\cat S)).\]
Since $\Fun^R(\cat C\op,\cat S)$ is equivalent to $\cat{C}$, this full subcategory is the category $\Fun^U(\mathcal{A},\cat{C})$.

Presentability of $\Fun^U(\mathcal A,\cat C)$ immediately follows from this description. 

The third part follows from the fact that the inclusion preserves limits (as limits commute with limits) and is accessible (as in any presentable category, for $\kappa$ large enough, $\kappa$-filtered colimits commute with $\kappa$-small limits), and from the adjoint functor theorem. 
\end{proof}

\subsection{Localization and Day convolution}

Now, we assume that $\mathcal{A}$ and $\cat{C}$ are symmetric monoidal categories. In that case, there is a symmetric monoidal structure on the functor category $\Fun(\mathcal{A},\cat{C})$ called the Day convolution structure. This symmetric monoidal structure enjoys the following universal property (see \cite[2.2.6.8]{HA})~:

The data of a lax symmetric monoidal functor $\cat B\to\Fun(\mathcal{A},\cat{C})$ is equivalent to the data of a lax symmetric monoidal functor $\cat B\times\mathcal{A}\to\cat{C}$.

\begin{nota}
We will let $\otimes^{Day}$ denote the Day tensor product.
\end{nota}

We make the following additional assumption on $\mathcal{A}$ and $U$.

\begin{assu}
The category $\mathcal{A}$ has a symmetric monoidal structure and for all diagrams $f_u:K_u^{\triangleleft}\to\mathcal{A}$ and all object $a$ of $\mathcal{A}$, the diagram $a\otimes f_u$ is also in $U$ (or equivalent to a diagram of $U$). 
\end{assu} 

\begin{theo}\label{locDay}
Let $\cat{C}$ be a symmetric monoidal presentable category.
\begin{enumerate}
\item The category $\Fun(\mathcal{A},\cat{C})$ equipped with the Day convolution tensor product is a symmetric monoidal presentable category.
\item There is a unique symmetric monoidal structure on $\Fun^U(\mathcal{A},\cat{C})$ such that the localization functor
\[\Fun(\mathcal{A},\cat{C})\to \Fun^U(\mathcal{A},\cat{C})\]
is symmetric monoidal.
\item The inclusion
\[\Fun^U(\mathcal{A},\cat{C})\to \Fun(\mathcal{A},\cat{C})\]
is lax symmetric monoidal.
\end{enumerate}
\end{theo}

To establish this, we will need a couple of lemmas.  First, recall the following result due to Glasman.

\begin{prop}{\cite[Section 3]{Glasman}}\label{SymMonYoneda}
Suppose $C$ is a small symmetric monoidal category, and let $\Fun(C\op,\cat S)$ have the Day convolution monoidal structure. Then there is a canonical symmetric monoidal structure on the Yoneda embedding $C\to \Fun(C\op,\cat S)$
\end{prop}
Using this and our assumption, we can conclude that our localization is compatible with the Day structure on $\Fun(\mathcal A,\cat S)$. 
\begin{prop}
The  localization $\Fun(\mathcal A,\cat S)\to \Fun^U(\mathcal A,\cat S)$ is compatible with the Day symmetric monoidal structure. Concretely, this means that if $f$ becomes an equivalence after localization, then $f\otimes^{Day}F$ does so too, for any $F\in \Fun(\mathcal A,\cat S)$.  
\end{prop}
\begin{proof}
Local equivalences are closed under colimits, and the Day convolution structure is compatible with colimits by \cite[2.13]{Glasman}, so we may assume $F=\map(P,-)$ for some $P\in \mathcal A$. 

For $f= \on{colim}_{i\in K_u} \map_{\mathcal{A}}({f_u(i)},-)\to \map_{\mathcal{A}}(f_u(\infty_u),-)$, since the Yoneda embedding is symmetric monoidal, $f\otimes^{Day}\map(P,-)$ is identified with 
\[\on{colim}_{i\in K_u} \map_{\mathcal{A}}({f_u(i)}\otimes P,-)\to \map_{\mathcal{A}}(f_u(\infty_u)\otimes P,-).\]
By assumption, the diagram $f_u\otimes P$ is (equivalent to a diagram) in $U$, so clearly this map is a local equivalence. Since our localization is accessible and the class of local equivalences is generated by maps $f$  of this form (see the proof of theorem \ref{presentability}), we are done.
\end{proof}

By \cite[2.2.1.9 and 4.1.7.4]{HA}, we automatically get Theorem \ref{locDay} for the special case $\cat C = \cat S$.  To extend to the general case, we use part (2) of Theorem \ref{presentability}. 

For this, we compare the two natural symmetric monoidal structures on $\Fun(\mathcal A,\cat C)$~: the Day convolution structure, and the one obtained from $\Fun(\mathcal A,\cat C)\simeq \Fun(\mathcal A,\cat S)\otimes \cat C$ (note that in $\cat{Pr}^L$, this tensor product is a coproduct of commutative algebras, so it gets a natural symmetric monoidal structure compatible with colimits in each variable).

\begin{lemm}\label{ExtDay}
The canonical functor $\Fun(\mathcal A,\cat S)\otimes\cat C\to \Fun(\mathcal A,\cat C)$ has a natural symmetric monoidal structure, where the right hand side has the Day convolution structure and the left hand side has the structure described above. It is therefore an equivalence of symmetric monoidal categories.
\end{lemm}

\begin{proof}
In $\cat{Pr}^L$, symmetric monoidal categories with a compatible tensor product correspond to commutative algebra objects, and therefore tensor product is simply the coproduct. Therefore, to equip the canonical functor $\Fun(\mathcal A,\cat S)\otimes \cat C\to \Fun(\mathcal A,\cat C)$ with a symmetric monoidal structure is equivalent to doing so for both $\Fun(\mathcal A,\cat S)\to \Fun(\mathcal A,\cat C)$ and $\cat C\to \Fun(\mathcal A,\cat C)$. 

Note that there is an essentially unique colimit preserving symmetric monoidal functor $\cat{S\to C}$. Recall the universal property of the Day convolution structure (see \cite[2.2.6.8]{HA}), which implies then that $\Fun(\mathcal A,\cat S)\times \mathcal A\to \cat S$ has a canonical lax symmetric monoidal structure, so that the composite $\Fun(\mathcal A,\cat S)\times \mathcal A\to \cat S\to \cat C$ does so too. 

In particular, this gives $\Fun(\mathcal A,\cat S)\to \Fun(\mathcal A,\cat C)$ a lax symmetric monoidal structure. The explicit formula for the Day tensor product, as well as the fact that $\cat S\to \cat C$ is symmetric monoidal and preserves colimits imply that this lax symmetric monoidal structure is actually a (strict) symmetric monoidal structure. 

Moreover, let $\odot$ temporarily denote the tensor product $\cat{C\times S\to C}$. Then the composite $\cat C\times \mathcal A\xrightarrow{\id\times \map(1,-)\,}\cat{C\times S}\xrightarrow{\odot} \cat C$ is a composite of functors that all have canonical lax symmetric monoidal structures, so it has one as well. Therefore, the corresponding functor $\cat C\to \Fun(\mathcal A,\cat C)$ does too. One may then similarly check (using the fact that the Yoneda embedding $\mathcal A^{op}\to \Fun(\mathcal A,\cat S)$ is symmetric monoidal) that it is also not only lax but actually a (strict) symmetric monoidal functor. 

Therefore we also get a symmetric monoidal functor $\Fun(\mathcal A,\cat S)\otimes \cat C\to \Fun(\mathcal A,\cat C)$, and it is easy to check that on the underlying categories, it is the canonical functor which is an equivalence.
\end{proof}

\begin{proof}[Proof of Theorem \ref{locDay}]
The localization $\Fun(\mathcal A,\cat C)\to \Fun^U(\mathcal A,\cat C)$ is identified with
\[\Fun(\mathcal A,\cat S)\otimes\cat C\to \Fun^U(\mathcal A,\cat S)\otimes\cat C.\]
In this second form, it acquires a symmetric monoidal structure, which means that the localization is compatible with the Day convolution structure on $\Fun(\mathcal A,\cat C)$. We can therefore again apply \cite[2.2.1.9 and 4.1.7.4]{HA}  and get a symmetric monoidal structure on $\Fun^U(\mathcal A,\cat C)$ with the desired properties. 
\end{proof}

\begin{coro}
The following statements hold.
\begin{enumerate}
\item There is a unique symmetric monoidal structure on $\Fun^{rex}(\Mod_\ZZ^{c,\omega}, \Mod_\ZZ)$ such that the localization functor
\[\Fun(\Mod_\ZZ^{c,\omega}, \Mod_\ZZ)\to \Fun^{rex}(\Mod_\ZZ^{c,\omega}, \Mod_\ZZ)\]
is symmetric monoidal.
\item The functor $Q$ is the unit of $\Fun^{rex}(\Mod_\ZZ^{c,\omega}, \Mod_\ZZ)$. 
\item The same holds for $\Fun^\oplus(\Ab^{\omega},\Mod_\ZZ)$ or $\Fun^\oplus(\Ab_\kappa,\Mod_\ZZ)$, and the restriction of $Q$ there. 
\end{enumerate}
\end{coro}

\begin{proof}
Statement (1) and the first half of statement (3) follow from Theorem \ref{locDay} using an appropriate choice of $U$ (using \ref{examples}).

By Lemma \ref{ExtDay} and Proposition \ref{SymMonYoneda}, the functor $\ZZ[-]$ is the monoidal unit of the category $\Fun(\Mod_\ZZ^{c,\omega}, \Mod_\ZZ)$. The functor $Q$ is the image of $\ZZ[-]$ under the localization functor, therefore it is the unit of the symmetric monoidal category $\Fun^{rex}(\Mod_\ZZ^{c,\omega}, \Mod_\ZZ)$.

The claim about $\Ab^{\omega},\Ab_\kappa$ follows in a similar way using Proposition \ref{QisAdd2}.
\end{proof}

\begin{lemm}\label{UniqueAlg}
Let $\cat{C}^{\otimes}$ be a symmetric monoidal category. Let $\mathbb{I}$ be the unit of $\cat{C}^{\otimes}$. Then $\mathbb{I}$ has a unique commutative algebra structure.
\end{lemm}

\begin{proof}
The more precise statement would be as follows. Let $\cat C^\otimes$ be a symmetric monoidal category with unit $\mathbb I$, and suppose $\mathbb I$ has an algebra structure with structure map $\eta: \mathbb I\to\mathbb I$. Then $\eta$ is an equivalence of commutative algebras.

To see that, note that by composing with $\cat C\to \mathrm{Ho}(\cat C)$, since this functor preserves algebra objects and reflects equivalences, we may assume $\cat C$ is a $1$-category, which we will accordingly denote by $C$. 

Then $\mathrm{End}_C(\mathbb I)$ is a commutative monoid under composition, by the Eckmann-Hilton argument, so it suffices to prove that $\eta$ has a retraction. We then have the following commutative diagram, where $\mu : \mathbb I\otimes \mathbb I\to \mathbb I$ is the multiplication of the algebra structure under consideration, the isomorphisms are the ones given by the symmetric monoidal structure on $C$, and $\alpha$ is defined by the commutative triangle :
\begin{center}

\begin{tikzcd}
\mathbb I\ar[r,"\sim"] \ar[dr,"\id"'] & \mathbb I\otimes \mathbb I \ar[r,"\id\otimes \eta"]\ar[d,"\sim"] & \mathbb I\otimes \mathbb I \ar[r,"\mu"] & \mathbb I \\
& \mathbb I  \ar[r,"\eta"] & \mathbb I \ar[u,"\sim"]\ar[ru,"\alpha"'] 
\end{tikzcd}
    
\end{center}
In particular, $\alpha\eta = \id_\mathbb I$, so we have a retraction, and we may thus conclude. 
\end{proof}

\begin{coro}\label{UniqueQ}
There is a unique commutative algebra structure on $Q$ in the category $\Fun(\Mod_\ZZ^{c,\omega},\Mod_\ZZ)$. The same holds for the restriction of $Q$ to $\Ab^{\omega}$ or $\Ab_\kappa$.
\end{coro}

\subsection{MacLane homology as a lax symmetric monoidal functor}

We are now equipped to give a construction of $\HML$ as a lax symmetric monoidal functor
\[\HML:\cat{Alg}_{\mathcal{AB}}(\Ab)\to \Mod_{\ZZ}.\]

\begin{cons}\label{monoidalHML}
Note that $Q$ is initial in $\Fun^{lax}(\Ab^\omega,\Mod_\ZZ)$  (where $\Fun^{lax}(\cat C,\cat D)$ denotes the category of lax symmetric monoidal functors $\cat C\to \cat D$) so it has a unique symmetric monoidal transformation $Q\to i$, where $i$ is the inclusion, and therefore, by \cite[4.8.1.10]{HA}, there is also a unique symmetric monoidal transformation $Q\to i$ as functors $\Ab\to \Mod_\ZZ$. The same can be said monoidally (that is, dropping symmetry), and so this transformation must agree with the monoidal augmentation that we defined in Section \ref{section : QC}. 

This symmetric monoidal augmentation therefore yields a symmetric monoidal augmention $Q\to i$ as functors $\Alg_{Ass}(\Ab)\to \Alg_{Ass}(\Mod_\ZZ)$, which is the one we used to define MacLane homology~: from it, we said that we could define the structure of a $Q(A)$-bimodule on any $A$-bimodule $M$. 

To define this more conceptually, one may say that the morphism $(Q(A),M)\to (A,M)$ in $\Alg_{\mathcal{AB}}(\Mod_\ZZ)$ is cartesian over $Q(A)\to A$. We may thus use Appendix \ref{appendix : monoidal fibration} to define it. By Theorem \ref{cartesianfib} (whose hypotheses are readily checked here),
\[\Fun^{lax}(\Alg_{\mathcal{AB}}(\Ab),\Alg_{\mathcal{AB}}(\Mod_\ZZ))\to \Fun^{lax}(\Alg_{\mathcal{AB}}(\Ab), \Alg_{Ass}(\Mod_\ZZ))\]
is a cartesian fibration; and we may thus lift $Q\to i$ to get our pointwise cartesian morphism $(Q(A),M)\to (A,M)$ functorially, and symmetric monoidally. In particular, we get a well-defined functor $(A,M)\mapsto (Q(A),M)$ from the category $\Alg_{\mathcal{AB}}(\Ab)$ to the category $\Alg_{\mathcal{AB}}(\Mod_\ZZ)$ which is lax symmetric monoidal. Since all other functors in the definition of $\HML$ are symmetric monoidal, this gives $\HML$ a lax symmetric monoidal structure. 
\end{cons}

\section{A symmetric monoidal equivalence}\label{section : sym mon equivalence}

The goal of this section is to prove the following theorem.

\begin{theo}\label{maintheoQ}
There is a symmetric monoidal equivalence $Q(-)\simeq \ZZ\otimes -$ in the category of functors $\Mod_\ZZ^c\to \Mod_\ZZ$
\end{theo}

The main tool that we shall use in order to prove this result is the following theorem.

\begin{theo}\label{moneq}
There is a symmetric monoidal equivalence of categories
\[\cat{Mod}_{\ZZ\otimes\ZZ}\to \Fun^{rex}(\Mod_\ZZ^{c,\omega},\Mod_\ZZ)\]
sending $M$ to the functor $M\otimes_{\ZZ}-$.
\end{theo}

We start by proving a non symmetric monoidal version of this theorem.

\begin{prop}
The functor
\[\Mod_{\ZZ\otimes \ZZ}\to\Fun^{rex}(\Mod_\ZZ^{c,\omega},\Mod_\ZZ)\]
sending $M$ to $M\otimes_\ZZ-$ is an equivalence.
\end{prop}

\begin{proof}
First, since $\Mod_\ZZ^c$ is compactly generated, the restriction map
\[\Fun^L(\Mod_\ZZ^c,\Mod_\ZZ)\to\Fun^{rex}(\Mod_\ZZ^{c,\omega},\Mod_\ZZ)\]
is an equivalence. Indeed, by \cite[5.3.5.10]{HTT}, this is true for filtered colimit-preserving functors on the left-hand side and arbitrary functors on the right hand side, and by \cite[5.5.1.9]{HTT}, this equivalence restricts to an equivalence between colimit-preserving functors on the right-hand side and finite colimit-preserving functors on the left-hand side.

By \cite[1.3.5.21]{HA}, the canonical $\mathrm t$-structure on $\Mod_\ZZ$ is right complete, and therefore $\Mod_\ZZ \simeq \lim(\dots \xrightarrow{\Omega} \Mod_\ZZ^c \xrightarrow{\Omega} \Mod_\ZZ^c)$, so that $\Mod_\ZZ$ is the presentable stabilization of $\Mod_\ZZ^c$. It follows that the restriction map
\[\Fun^L(\Mod_\ZZ,\Mod_\ZZ)\to\Fun^{rex}(\Mod_\ZZ^{c,\omega},\Mod_\ZZ)\]
is an equivalence. The desired result then follows from \cite[7.1.2.4]{HA}.
\end{proof}

\begin{rem}
One can prove in a similar way that $\Fun^{\oplus}(F_\ZZ,\Mod_\ZZ)\simeq \Mod_{\ZZ\otimes\ZZ}$. From this we can give more details about Remark \ref{remkbadz}. Indeed, we have $\Fun^\times(F_\ZZ\op,\Ab)\simeq \Ab$ so that $\Fun^\times(F_\ZZ\op,\Ch_*(\ZZ))\simeq \Ch_*(\ZZ)$ and in particular
\[W^{-1}\Fun^\times(F_\ZZ\op,\Ch_*(\ZZ))\simeq \Mod_\ZZ \not\simeq \Mod_{\ZZ\otimes\ZZ}\simeq \Fun^\oplus(F_\ZZ,\Mod_\ZZ) \simeq \Fun^\times(F_\ZZ\op, \Mod_\ZZ),\]
where the last equivalence follows from the fact that $\hom(-,\ZZ): F_\ZZ\op\to F_\ZZ$ is an equivalence. 
\end{rem}

\begin{prop}
There is a lax symmetric monoidal structure on the functor $\Mod_{\ZZ\otimes \ZZ}\to \Fun^{rex}(\Mod_\ZZ^{c,\omega},\Mod_\ZZ)$
\end{prop}
\begin{proof}
We write our functor as a composite 
\[\Mod_{\ZZ\otimes \ZZ}\to \Fun(\Mod_\ZZ^{c,\omega},\Mod_\ZZ)\to \Fun^{rex}(\Mod_\ZZ^{c,\omega},\Mod_\ZZ).\]
By the work of Section \ref{sec : sym mon structure}, the second arrow has a canonical symmetric monoidal structure.

We shall give the first map a lax symmetric monoidal structure. By \cite[2.2.6.8]{HA}, this amounts to giving the adjoint functor
\[\Mod_{\ZZ\otimes \ZZ}\times \Mod_\ZZ^{c,\omega} \to \Mod_\ZZ\]
a lax symmetric monoidal structure. 

We then note that this functor can be written as a composite
\[\Mod_{\ZZ\otimes \ZZ}\times \Mod_\ZZ^{c,\omega}\to \Mod_{\ZZ\otimes \ZZ}\times \Mod_\ZZ\to \Mod_\ZZ,\]
where the first functor has a natural symmetric monoidal structure. The second functor is itself the composite
\[\Mod_{\ZZ\otimes \ZZ}\times \Mod_\ZZ \to \Mod_{\ZZ\otimes \ZZ}\otimes \Mod_\ZZ\to \Mod_\ZZ.\]
The first functor is lax monoidal. Recall, from \cite[4.8.5.16]{HA}, that $\Mod_\ZZ\otimes \Mod_\ZZ\simeq \Mod_{\ZZ\otimes \ZZ}$ via $(M,N)\mapsto M\otimes N$. It follows that we only need to give a lax symmetric monoidal structure to the functor $\Mod_\ZZ\otimes \Mod_\ZZ\otimes \Mod_\ZZ\to \Mod_\ZZ$ sending $ (M,N,A)$ to $M\otimes U(N\otimes_\ZZ A)$ where $U: \Mod_\ZZ\to \Sp$ denotes the forgetful functor. This has a clear lax symmetric monoidal structure. Indeed, the functor $\otimes_\ZZ : \Mod_\ZZ\otimes \Mod_\ZZ \to \Mod_\ZZ$ is symmetric monoidal, the functor $U: \Mod_\ZZ\to \Sp$ is lax symmetric monoidal, and the equivalence $\Mod_\ZZ\otimes \Sp\xrightarrow{\simeq} \Mod_\ZZ$ is symmetric monoidal. 
\end{proof}

The claim of Theorem \ref{moneq} is now that this is actually a symmetric monoidal equivalence.

\begin{proof}[Proof of Theorem \ref{moneq}]
We let $F$ denote our functor for this proof. We need to prove two things. First, that the canonical map $Q\to F(\ZZ\otimes \ZZ)$ is an equivalence and second that for any $M,N\in~\Mod_{\ZZ\otimes \ZZ}$, the canonical map $F(M)\otimes_{rex}^{Day} F(N)\to F(M\otimes_{\ZZ\otimes\ZZ}N)$ is an equivalence, where $\otimes_{rex}^{Day}$ denotes the tensor product on $\Fun^{rex}$.

We already know that $F$ is an equivalence on the underlying categories and that both tensor products commute with colimits and desuspensions. It follows that, in order to prove the second point, we can restrict to $M=N=\ZZ\otimes\ZZ$. Then, for this one, various compatibility relations imply that it actually suffices to check the first point, that is, that $Q\to F(\ZZ\otimes \ZZ)$ is an equivalence. 

For this, we use Lemma \ref{UniqueAlg}. Indeed, since $F$ was given a lax symmetric monoidal structure, $Q\to F(\ZZ\otimes\ZZ)$ is the unit map of a commutative algebra structure on $F(\ZZ\otimes \ZZ)$. Therefore, if we know, for some other reason, that $F(\ZZ\otimes \ZZ)$ is equivalent to $Q$, then it will follow that this unit map must be an equivalence. 

So we are reduced to proving that $F(\ZZ\otimes \ZZ)$ is equivalent to $Q$, asbtractly. For this, we use the Yoneda lemma, as well as Theorem \ref{QisAdd}. Indeed, we have, naturally in $M\in \Mod_{\ZZ\otimes \ZZ}$~: 
\[\map_{\Mod_{\ZZ\otimes\ZZ}}(\ZZ\otimes \ZZ, M)\simeq \map_{\Sp}(\SS, M)\simeq \Omega^\infty M\]
(we have suppressed the forgetful functor $\Mod_{\ZZ\otimes \ZZ}\to \Sp$ from the notation).

We also have, naturally in $M$~: 
\begin{align*}
\map_{\Fun^{rex}(\Mod_\ZZ^{c,\omega},\Mod_\ZZ)}(Q,F(M))&\simeq \map_{\Fun(\Mod_\ZZ^{c,\omega},\Mod_\ZZ)}(\ZZ[-],F(M))\\
&\simeq \map_{\Fun(\Mod_\ZZ^{c,\omega},\cat{Sp})}(\Sigma^\infty_+\map(\ZZ,-),F(M))\\
&\simeq \map_{\Fun(\Mod_\ZZ^{c,\omega},\cat S)}(\map(\ZZ,-),\Omega^\infty \circ F(M))\\ 
&\simeq \Omega^\infty(F(M)(\ZZ)) \\
&\simeq \Omega^\infty M,
\end{align*}
where we again suppressed the forgetful functor from the notation, as well as the inclusion functor $\Fun^{rex}(\Mod_\ZZ^{c,\omega},\Mod_\ZZ)\to \Fun(\Mod_\ZZ^{c,\omega},\Mod_\ZZ)$. In this equivalence, the first line uses Theorem \ref{QisD1}, and in the second to last line, we use the Yoneda lemma applied to $\Mod_\ZZ^{c,\omega}$.  It follows, by the Yoneda lemma, that $F(\ZZ\otimes \ZZ)\simeq Q$. By what we said above, this concludes the proof. 
\end{proof}

We may now prove Theorem \ref{maintheoQ}:

\begin{proof}[Proof of Theorem \ref{maintheoQ}]
From Theorem \ref{moneq}, we deduce that the image of $\ZZ\otimes\ZZ$ via this equivalence must coincide with the unit of $\Fun^{rex}(\Mod_\ZZ^{c,\omega},\Mod_\ZZ)$ namely $Q$. Explicitly, this is saying that the functor $M\mapsto \ZZ\otimes M$ is equivalent to $Q$ in the category $\Fun^{rex}(\Mod_\ZZ^{c,\omega},\Mod_\ZZ)$. 

As $\Mod_\ZZ^c \simeq \mathrm{Ind}(\Mod_\ZZ^{c,\omega})$, we can apply \cite[Proposition 4.8.1.10]{HA} to get that the category of filtered-colimit preserving lax symmetric monoidal functors $\Mod_\ZZ^c\to \Mod_\ZZ$ is equivalent to that of lax symmetric monoidal functors $\Mod_\ZZ^{c,\omega}\to \Mod_\ZZ$, the equivalence being given by restriction. Since $Q$ and $\ZZ\otimes-$ have the same restriction to $\Mod_\ZZ^{c,\omega}$, it suffices to check that they both preserve filtered colimits.

Obviously the functor $M\mapsto \ZZ\otimes M$ preserves all colimits. This is also the case for $Q$ by definition (see Definition \ref{defiQ}).
\end{proof}

\section{Comparison with Richter's work}\label{section : Richter}

In this section we show that the lax symmetric monoidal structure that we have on the functor $Q:\Ab\to\Mod_{\ZZ}$ coincides with the one constructed by Richter in \cite{richter}. Let $(C,\otimes,1_{C})$ be either $\Ab^\omega$ or $\Ab_\kappa$ with its usual symmetric monoidal structure.

We use \cite{hinich} to explain how to construct, from the $E_\infty$-algebra structure on $Q$ constructued by Richter a commutative algebra structure in the sense of \cite{HA} on $Q$, which we will call $Q^R$. We first fix a notation. Given an operad $P$ in the category $\Ch(\ZZ)$ and a symmetric monoidal $1$-category $\ucat{M}$ enriched over $\Ch(\ZZ)$, we denote by $\Alg^{st}_P(\ucat{M})$ the category of $P$-algebras in $\ucat{M}$. Explicitly, a $P$-algebra in $\ucat{M}$ is the data of an object $A$ of $\ucat{M}$ and a map of dg-operads from $P$ to the operad of endomorphisms of $A$. In \cite{hinich}, Hinich introduces the notion of homotopically sound dg-operad. The precise definition is irrelevant here, the only important fact is that, if $P$ is homotopically sound, the category $\Alg^{st}_P(\Ch(\ZZ))$ admits a projective model structure (weak equivalences and fibrations are colorwise) and  the underlying $\infty$-category of this model category only depends on the quasi-isomorphism type of $P$. Let us also mention that cofibrant dg-operads are homotopically sound (see \cite[Proposition 2.3.2]{hinich}) which implies that any dg-operad can be replaced by a quasi-isomorphic homotopically sound operad.

  We now recall the main theorem of \cite{hinich}. 

\begin{theo}\label{theo : main Hinich}
Let $P$ be a simplicial operad (that we view as an $\infty$-operad) and let $R\to C_*(P)$ be a homotopically sound replacement. Then, there is an equivalence
\[W^{-1}\Alg_R^{st}(\Ch_*(\ZZ)) \xrightarrow{\simeq} \Alg_{P}(\Mod_\ZZ).\]
\end{theo}

\begin{cons}\label{cons : Richter}
Richter constructs an $E_\infty$-dg-operad $\mathcal O$ together with an action of $\mathcal O$ on the functor $Q$. More precisely she constructs a $(\mathcal O\times C)$-algebra in $\Ch_*(\ZZ)$ where $C$ here is viewed as the operad underlying a symmetric monoidal ($1$-)category, and the map $\mathcal O\times C\to \Ch_*(\ZZ)$ is $\Ch_*(\ZZ)$-enriched in the $\mathcal O$-coordinate. 

By that, we more precisely mean that for each $n$ and each $(A_1,...,A_n;A)\in C$, we have a map $L: \mathcal O(n)\times \mathrm{Mult}(A_1,...,A_n;A)\to \mathrm{Mult}(Q(A_1),...,Q(A_n);Q(A))$ such that at a fixed $f\in \mathrm{Mult}(A_1,...,A_n;A)$, the map $g\mapsto L(g,f) $ is a chain map. 

Equivalently, this is a $\mathcal O\otimes_{\ZZ} \ZZ[C]$-algebra in $\Ch_*(\ZZ)$, where $\mathcal O\otimes_{\ZZ} \ZZ[C]$ is the operad with colors the colors of $C$, and 
\[\mathrm{Mult}_{O\otimes_{\ZZ} \ZZ[C]}(A_1,...,A_n;A) = \mathcal{O}(n)\otimes_{\ZZ} \ZZ[\mathrm{Mult}_C(A_1,...,A_n;A)]\]
(note that the multi-operation spaces in $C$ are just sets). 

 Let $C_*$ denote the singular chains functor. Note that because $C$ is a discrete simplicial operad, for any simplicial operad $E$, we have an isomorphism of operads in chain complexes
\[C_*(E\times C)\cong C_*(E)\otimes_{\ZZ} \ZZ[C].\]
So if we take $E$ to be a cofibrant $E_\infty$-operad in simplicial sets, then $C_*(E)$ will be a cofibrant dg-operad. In particular, there exists a quasi-isomorphism of operads $C_*(E)\to \mathcal O$ and thus a dg-operad map $C_*(E\times C)\to \mathcal O\otimes_{\ZZ} \ZZ[C]$. So Richter constructs a $C_*(E\times C)$-algebra in $\Ch_*(\ZZ)$. Finally, let us take a homotopically sound replacement $\mathcal R\to C_*(E\times C)$ in the sense of \cite{hinich}. Then, using Theorem \ref{theo : main Hinich}, we can view Richter's construction as constructing an object in $\Alg_{E\times C}(\Mod_{\ZZ})$ whose underlying functor $C\to \Mod_{\ZZ}$ is equivalent to $Q$. We call this object $Q^R$. 

Note that $Q^R\in \Alg_{E\times C}(\Mod_\ZZ)$, so by the defining universal property of Day convolution (see \cite[2.2.6.8]{HA}), we  may view $Q^R$ as an object of $\in \Alg_E(\Fun(C,\Mod_\ZZ))$. Since $E$ is equivalent to the commutative operad, we have an equivalence 
\[\Alg_E(\Fun(C,\Mod_\ZZ))\simeq \Alg_{\mathcal Com}(\Fun(C,\Mod_\ZZ)).\]
Seen through this equivalence, Richter's construction yields a commutative algebra in $\Fun(C,\Mod_\ZZ)$ that we still denote by $Q^R$ and whose underlying functor $C\to \Mod_{\ZZ}$ is equivalent to $Q$. 
\end{cons}

From this construction, we can extract the following proposition.

\begin{prop}
Let $C=\Ab^{\omega}$ or $C=\Ab_{\kappa}$ for any cardinal $\kappa$. Then Richter's construction produces a lax symmetric monoidal functor $Q^R:C\to \Mod_\ZZ$ which is equivalent to $Q$.
\end{prop}

\begin{proof}
Construction \ref{cons : Richter} produces a lax symmetric monoidal functor $Q^R$ whose underlying functor is $Q$. The result then follows from Corollary \ref{UniqueQ}.
\end{proof}

This proposition should be sufficient for any application as we can always choose a $\kappa$ big enough so that all the objects that we are interested in live in $\Ab_\kappa$. 

We can also obtain a statement that is independant of a choice of $\kappa$ using the theory of Grothendieck universes. Let $V_0$ and $V_1$ be two Grothendieck universes with $V_0\in V_1$. We shall denote by $\Ab(V_0)$ (resp. $\Ab(V_1)$) the category of abelian groups in $V_0$ (resp. in $V_1$). Likewise, we denote by $\Ch(\ZZ)(V_0)$ (resp. $\Ch(\ZZ)(V_1)$) the chain complexes in $\Ab(V_0)$ (resp. $\Ab(V_1)$). Finally, we denote by $\Mod_\ZZ(V_0)$ (resp. $\Mod_\ZZ(V_1)$ the localization of $\Ch(\ZZ)(V_0)$ (resp. $\Ch(\ZZ)(V_1)$) at the quasi-isomorphisms.

\begin{prop}\label{prop : Richter kappa}
Let $\Ch_*(\ZZ)(V_0)\to \Ch_*(\ZZ)(V_1)$ be the canonical fully-faithful inclusion. Then the induced functor $\Mod_\ZZ(V_0)\to \Mod_\ZZ(V_1)$ is fully faithful. 
\end{prop}

\begin{proof}
Since both categories are stable, it suffices to check the claim on homotopy categories. 

We can then use the projective model structure on $\Ch_*(\ZZ)(V_0)$ to replace any pair of objects with a pair of bifibrant objects, for which it is clear that the hom-sets in the homotopy category agree (because in the case of $\Ch_*(\ZZ)$, homotopies in the model category are the same as the usual chain homotopies).

We just need to make sure that cofibrant objects in $\Ch_*(\ZZ)(V_0)$ are cofibrant in $\Ch_*(\ZZ)(V_1)$. This is clear using the fact that in cofibrantly generated model categories, cofibrant objects are retracts of cell complexes, and that in the case of $\Ch_*(\ZZ)$, the generating cofibrations do not depend on the universe. 
\end{proof}

We can then apply the previous work to $C= \Ab(V_0)$, we then have an equivalence $Q^R\simeq Q$ as commutative algebras in $\Fun(\Ab(V_0), \Mod_\ZZ(V_1))$, so as lax symmetric monoidal functors $\Ab(V_0)\to \Mod_\ZZ(V_1)$. However, they both land in the full subcategory $\Mod_\ZZ(V_0)$ (the explicit model of the $Q$-construction clearly shows this) so that they are equivalent as lax symmetric monoidal functors $\Ab(V_0)\to \Mod_\ZZ(V_0)$. 

Taking $V_0$ to be our universe of interest we obtain the following proposition.

\begin{prop}\label{prop : Richter without kappa}
Richter's construction endows $Q$ with a lax symmetric monoidal structure $Q^R$ on $\Ab$. Moreover, we have an equivalence $Q^R\simeq Q$ in the category of lax symmetric monoidal functors $\Ab\to\Mod_\ZZ$.
\end{prop}

\begin{proof}
By Proposition \ref{prop : Richter kappa}, they are equivalent as lax monoidal functors on $\Ab^{\omega}$. Just as in the proof of Theorem \ref{maintheoQ}, we can use \cite[Proposition 4.8.1.10.]{HA} and the fact that $\Ab\simeq \mathrm{Ind}(\Ab^{\omega})$ to conclude that $Q^R\simeq Q$ in the category of lax monoidal functors from $\Ab$ to $\Mod_{\ZZ}$.
\end{proof}

\section{Conclusion}\label{section : conclusion}

We can now prove our main Theorem. We have the category $\Alg_{\mathcal{AB}}(\Ab)$ of associative algebras and bimodules in the category of abelian groups. This category has a symmetric monoidal structure given by the formula
\[(A,M)\otimes(B,N)=(A\otimes_{\ZZ}B,M\otimes_{\ZZ}N).\]
The map of commutative algebras $\alpha:\SS\to \ZZ$ induces a lax symmetric monoidal functor $\alpha^*:\Alg_{\mathcal{AB}}(\Ab)\to \Alg_{\mathcal{AB}}(\cat{Sp})$. To simplify notations, we will simply denote by $(A,M)$ the value of this functor on an object $(A,M)$ of $\Alg_{\mathcal{AB}}(\Ab)$. In particular,
\[(A,M)\mapsto \THH(A,M)\]
defines a lax symmetric monoidal functor from $\Alg_{\mathcal{AB}}(\Ab)$ to $\cat{Sp}$.

\begin{theo}\label{theo : main}
There is an equivalence of lax symmetric monoidal functors from $\Alg_{\mathcal{AB}}(\Ab)$ to $\cat{Sp}$
\[\HML(A,M)\simeq\THH(A,M).\]
\end{theo}

\begin{proof}
Using the morphism of commutative algebras $\alpha: \SS\to\ZZ$, we can construct the category $\Alg_{\mathcal{AB}}(\SS,\ZZ)$ as in Paragraph \ref{subsection : base change formula}. There is a lax symmetric monoidal functor
\[\gamma:\Alg_{\mathcal{AB}}(\Mod_{\ZZ})\to \Alg_{\mathcal{AB}}(\SS,\ZZ)\]
which is right adjoint to the functor $\overline{\alpha}_!$. We still write $\gamma$ for the restriction of this functor to the subcategory $\Alg_{\mathcal{AB}}(\Ab)\subset\Alg_{\mathcal{AB}}(\Mod_{\ZZ})$.

We can precompose the base change formula (Theorem \ref{theo : base change}) with the functor $\gamma$, and we get an equivalence
\begin{equation}\label{eqn : main result}
\THH(\overline{\alpha}^*\gamma(A,M))\simeq \alpha^*\HH_{\ZZ}(\overline{\alpha}_!\gamma(A,M))
\end{equation}
in the category of lax symmetric monoidal functors from $\Alg_{\mathcal{AB}}(\Ab)$ to spectra.

By composition of adjunctions, the composite $\overline{\alpha}^*\gamma$ is the right adjoint to the functor $\overline{\alpha}_!\beta$ which is equal to the functor $\alpha_!:\Alg_{\mathcal{AB}}(\Mod_{\SS})\to\Alg_{\mathcal{AB}}(\Mod_{\ZZ})$. It follows that this composite coincides with the functor $\alpha^*: \Alg_{\mathcal{AB}}(\Mod_{\ZZ})\to\Alg_{\mathcal{AB}}(\Mod_{\SS})$ In particular, the left-hand side of equivalence \ref{eqn : main result} is simply the functor $(A,M)\mapsto\THH(A,M)$

Now, we claim that we have an equivalence
\[\overline{\alpha}_!\gamma(A,M)\simeq(Q(A),M)\]
in the category of lax symmetric monoidal functors from $\Alg_{\mathcal{AB}}(\Ab)$ to $\Alg_{\mathcal{AB}}(\Mod_{\ZZ})$. This claim will be proved at the end of the section. Admitting it for the moment, the right-hand side of equivalence \ref{eqn : main result} can be identified with $\HML(A,M)$ as desired.
\end{proof}

In particular, we deduce the following corollary.

\begin{coro}
There is an equivalence of commutative algebras in $\Sp$
\[\HML(R)\simeq \THH(R)\]
that is functorial in the discrete commutative ring $R$.
\end{coro}

\begin{proof}
There is an obvious functor
\[\Alg_{\mathcal{C}om}(\Ab)\to \Alg_{\mathcal{C}om}(\Alg_{\mathcal{AB}}(\Ab))\]
that sends a discrete commutative ring $R$ to the pair $(R,R)$. Composing this functor with the equivalence in the Theorem above, we obtain the desired result.
\end{proof}

The rest of this section will be devoted to the unjustified claim in the proof of Theorem \ref{theo : main}. Recall from Construction \ref{monoidalHML} that the lax symmetric monoidal functor $(A,M)\mapsto (Q(A),M)$ is defined via a cartesian lift of $Q(A)\to A$ along 
\[\Fun^{lax}(\Alg_{\mathcal{AB}}(\Ab),\Alg_{\mathcal{AB}}(\Mod_\ZZ))\to \Fun^{lax}(\Alg_{\mathcal{AB}}(\Ab),\Alg_{Ass}(\Mod_\ZZ)).\]
So to prove that statement, we only need to prove that the co-unit $\overline\alpha_!\gamma(A,M)\to (A,M)$ is such a cartesian lift (more precisely, we will prove that it's a cartesian lift of $\ZZ\otimes A\to A$, which is equivalent to what we need by theorem \ref{maintheoQ}). To prove this, we prove first a categorical result. 
\newcommand{\A}{\cat A}
\newcommand{\B}{\cat B}
\newcommand{\C}{\cat C}
\newcommand{\D}{\cat D}
\newcommand{\E}{\cat E}

\begin{prop}
Let $\A,\B,\C,\E$ be categories with a pullback square as follows~: 
\[\xymatrix{\E \ar[r]^p \ar[d]_q & \C \ar[d]^g \\ 
\B\ar[r]_b & \A}\]
Assume that $g$ is a cartesian fibration, and that $b$ has a right adjoint $b^*$.

Then $p$ also has a right adjoint $p^*$, the co-unit $pp^*\to id$ is pointwise $g$-cartesian and $qp^* \simeq b^*g$.
\end{prop}

\begin{proof}
Let $\epsilon : bb^*\to id$ denote the co-unit of the adjunction $b\dashv b^*$; and consider $\epsilon g:~bb^*g\to~ g$. This is a morphism in $\Fun(\C,\A)$ and its codomain is $g_*(id_\C)$, so, because
\[g_*:\Fun(\C,\C)\to \Fun(\C,\A)\]
is a cartesian fibration, it has a cartesian lift $\delta \to id$. 

Moreover, $g\delta \simeq bb^*g$ by definition, so we may define a functor $m: \C\to \E$ by the requirements that $pm\simeq \delta$ and $qm \simeq b^*g$. 

We will now prove that $m$ is right adjoint to $p$ and that the co-unit of this adjunction is identified with the lift $pm\simeq \delta \to id$. This will prove all the statements at once. 

Let us denote by $a:\E\to\A$ the composition $a=g\circ p=b\circ q$. Naturally in $e\in \E,c\in \C$, we have
\begin{align}\label{pullbackmap}
\map_\E(e,m(c)) \simeq \map_\C(p(e),pm(c))\times_{\map_\A(a(e),am(c))} \map_\B(q(e),qm(c)).
\end{align}
Observe that, since $pm\simeq \delta$, $\map_\C(p(e),pm(c))\simeq \map_\C(p(e),\delta(c))$ and so by $c$-cartesian-ness of $\delta \to id$, the latter is equivalent to $\map_\C(p(e),c)\times_{\map_\A(a(e),g(c))}\map_\A(g(p(e)),g(\delta(c)))$, where the map on the second factor is just $g_* : \map_\C(p(e),\delta(c))\to \map_\A(g(p(e)),g(\delta(c)))$.

In particular, this is the same as the map $\map_\C(p(e),pm(c))\to \map_\A(a(e),am(c))$ which defines the fiber product in the equivalence \ref{pullbackmap}, so that we actually get 
\[\map_\E(e,m(c)) \simeq \map_\C(p(e),c)\times_{\map_\A(a(e),g(c))}\map_\B(q(e),qm(c)).\]
Note moreover that the map $\map_\B(q(e),qm(c))\to \map_\A(a(e),g(c))$ is the composite 
\[\map_\B(q(e),qm(c)) \xrightarrow{b_*} \map_\A(a(e),am(c)) \to \map_\A(a(e),g(c)),\]
where the second arrow is given by postcomposition with $am(c) \simeq g\delta(c)\xrightarrow{g(\delta(c) c) } g(c)$.

In particular, using the definition of $\delta\to id$, this is identified with $b^*bg(c) \xrightarrow{\epsilon_{g(c)}} g(c)$ (naturally in $c$), and so the map $\map_\B(q(e),qm(c))\to \map_\A(a(e),g(c))$ is, all in all, identified with 
\[\map_\B(q(e),b^*bg(c)) \xrightarrow{b_*} \map_\A(a(e),bb^*bg(c)) \to \map_\A(a(e),bg(c)) = \map_\A(b q(e), bg(c))\]
which is known to be an equivalence, because it's the natural adjunction equivalence between $b$ and $b^*$.

In particular, we get a natural map $\map_\E(e,m(c)) \to \map_\C(p(e),c)$ which is an equivalence. Tracing through the equivalences, we see that this natural map can be described as
\[\map_\E(e,m(c))\xrightarrow{p_*} \map_\C(p(e),pm(c)) = \map_\C(p(e),\delta(c)) \xrightarrow{\delta(c)\to c} \map_\C(p(e),c).\]
This is precisely the claim that $p\dashv m$ and that the co-unit is identified with $\delta\to id$, which is what we wanted.
\end{proof}

We can now finally give the proof of the claim.

\begin{proof}[Proof of the claim]
Recall that we have the following pullback square
\[
\xymatrix{\Alg_{\mathcal{AB}}(\SS,\ZZ) \ar[r]^{\overline\alpha_!} \ar[d]^\kappa & \Alg_{\mathcal{AB}}(\Mod_\ZZ) \ar[d]^{res} \\ 
\Alg_{Ass}(\Sp)\ar[r]_{\alpha_!} & \Alg_{Ass}(\Mod_\ZZ)}
\]
and $\alpha^*$ is right adjoint to $\alpha_!$, $res$ is a cartesian fibration. $\gamma$ is right adjoint to $\overline\alpha_!$, so it follows from the previous proposition that the co-unit $\overline\alpha_!\gamma(A,M)\to (A,M)$ is $res$-cartesian, which means that the bimodule component of $\overline\alpha_!\gamma(A,M)$ is just the pullback of $M$ along $res\circ~\overline\alpha_!\gamma(A,M)\to~A$.

But now, $res\circ \overline\alpha_! \gamma \simeq \alpha_! \kappa \gamma$, and by the above, $\kappa\gamma \simeq \alpha^*\circ res$, so 
\[res\circ \overline\alpha_!\gamma(A,M) \simeq \alpha_!\alpha^*A\simeq \ZZ\otimes A\]
and the map $\alpha_!\alpha^*A\to A$ is just the co-unit of $\alpha_!\dashv \alpha^*$. 

This is not quite symmetric monoidal so far, but using Appendix B, we can make this identification symmetric monoidal. Indeed, the co-unit $\overline\alpha_!\gamma(A,M)\to (A,M)$ is a pointwise cartesian morphism which is also a transformation of lax symmetric monoidal functors, so it suffices to prove that $res\circ \overline\alpha_!\gamma (A,M)\simeq \ZZ\otimes A$ is a symmetric monoidal identification. 

But now the identification $res\circ \overline\alpha_!\gamma\simeq \alpha_!\alpha^*res$ relies on the identification $\kappa\gamma \simeq \alpha^*res$, so we just need to show that we can make this symmetric monoidal. 

Let us use the notations of the previous proof. By Theorem \ref{cartesianfib}, we can chose $\delta \to id$ to be a symmetric monoidal transformation of lax symmetric monoidal functors; and we can thus also make $m$ lax symmetric monoidal with, by definition, a symmetric monoidal identification $qm \simeq b^*g$ (here, $\kappa m \simeq \alpha^*\circ res$) and $pm\simeq \delta$ (note that the forgetful functor from symmetric monoidal categories to categories preserves pullbacks).

Therefore, what we need to do is show that the equivalence $\kappa m \simeq \kappa \gamma$ (the identification that comes from the fact that both $\gamma$ and $m$ are left adjoints to $\overline\alpha_!$) can be made symmetric monoidal. But this identification can be seen as the composite $m\to \gamma\overline\alpha_! m \to \gamma$ where the first map is the unit of $\overline\alpha_! \dashv \gamma$, and the second map is the co-unit of $\overline\alpha_!\dashv m$, which is identified with $\delta \to id$ and is therefore symmetric monoidal. So then $m$ is equivalent to $\gamma$ as symmetric monoidal functors (where $m$ is this ``new'' adjoint that we defined in the proof), and so the whole identification $res\circ \overline\alpha_!\gamma(A,M)\simeq \ZZ\otimes A$ can be made symmetric monoidally. By cartesian-ness, it follows that we have an equivalence
\[\overline\alpha_!\gamma(A,M)\simeq (Q(A),M)\]
of symmetric monoidal functors of $(A,M)$, which is what was claimed. 
\end{proof}

\appendix
\section{On a remark in \cite{FPSVW}}\label{appendix : remark}

In this appendix, we make a remark about remark (3.9) in \cite{FPSVW}. Theorem \ref{maintheoQ} implies that for any ring $R$, we have an equivalence of $\ZZ$-algebras $Q(R)\simeq \ZZ\otimes R$. In particular $Q(\ZZ)\simeq \ZZ\otimes \ZZ$ and so $Q(\ZZ)\otimes_\ZZ R \simeq \ZZ\otimes R$. 

Note, however, that in this second equivalence, the $\ZZ$-algebra structure on $\ZZ\otimes R$ comes from the $R$-factor, not the $\ZZ$-factor; whereas in the equivalence $Q(R)\simeq \ZZ\otimes R$, the $\ZZ$-algebra structure on $\ZZ\otimes R$ comes from the $\ZZ$-factor.

In particular, our main result does \textit{not} imply that $Q(R)\simeq Q(\ZZ)\otimes_\ZZ R$ as $\ZZ$-algebras in general. This formula is known to hold in some special cases, e.g. when $R$ is a generalized monoid algebra (this is for instance given as exercise E.13.4.1. in \cite{loday}).

In fact, the existence of an equivalence $Q(R)\simeq Q(\ZZ)\otimes_\ZZ R$ is equivalent to requiring that the two natural $\ZZ$-algebra structures on $\ZZ\otimes R$ are equivalent. The ``recent calculations'' that are mentioned in that remark are therefore probably related to the latter question, but in fact their inequivalence does not contradict the conjecture. 

We now give an example to show that the two $\ZZ$-algebra structures on $\ZZ\otimes R$ are not always equivalent.

\begin{prop}
The two natural $\ZZ$-algebra structures on $\ZZ\otimes \FF_2$ are not equivalent. 
\end{prop}

\begin{rem}
Here, the claim is not that the identity cannot be lifted to a $\ZZ$-algebra map, but that there is no ``abstract'' equivalence of $\ZZ$-algebras between the two. 
\end{rem}

\begin{proof}
Suppose they were equivalent. Then tensoring over $\ZZ$ with $\FF_2$ would also yield equivalent algebras, i.e. we would have an equivalence of $\ZZ$-algebras~: 
\[\FF_2\otimes \FF_2 \simeq \ZZ\otimes (\FF_2\otimes_\ZZ \FF_2)\]
As an associative algebra, we have $\FF_2\otimes_\ZZ \FF_2\simeq \FF_2[\epsilon]$ with $|\epsilon|=1, \epsilon^2= 0$, in particular, as spectra $\FF_2\otimes_\ZZ \FF_2 \simeq \FF_2\oplus \Sigma \FF_2$ and so the inclusion $\FF_2\otimes_\ZZ \FF_2\to \ZZ\otimes(\FF_2\otimes_\ZZ \FF_2)$ is injective on $\pi_1$. In particular, the right-hand side has a nonzero nilpotent class in $\pi_1$. On the other hand, the left-hand side has the dual mod-$2$ Steenrod algebra as its homotopy groups which is a polynomial algebra (main result of \cite{milnor}) and so has no nonzero nilpotents. The result follows.
\end{proof}

\begin{rem}
A similar result holds for $\FF_p$, with $p$ odd. In that case, any class in odd degree squares to zero, so the argument does not immediately work. However, one can argue as follows. By naturality, and for degree reasons the $p$th Massey power of $\epsilon$ must vanish in the right-hand side, whereas it does not in the dual mod-$p$ Steenrod algebra. Massey products only depend on the $E_1$-structure of the algebras, so the result follows as well. 
\end{rem}

\section{Monoidal fibrations}\label{appendix : monoidal fibration}

\newcommand{\Fin}{\mathrm{Fin}}
This appendix is devoted to proving the following theorem: 
\begin{theo}\label{cartesianfib}
Let $\C^\otimes,\D^\otimes,\E^\otimes$ be symmetric monoidal categories, $p^\otimes : \D^\otimes\to \E^\otimes$ a symmetric monoidal functor such that the underlying functor $p: \D\to \E$ is a Cartesian fibration.

Then $\Fun^{lax}(\C,\D)\to \Fun^{lax}(\C,\E)$ is a Cartesian fibration too and the Cartesian morphisms are precisely the ones that are pointwise Cartesian. 
\end{theo}

\begin{proof}
$p:\D\to \E$ is a Cartesian fibration, therefore for each $\langle n\rangle \in \Fin_*$, $p^\otimes_{\langle n \rangle}:\D^\otimes_{\langle n\rangle}\to \E^\otimes_{\langle n\rangle}$, which is identified with $\D^n\to \E^n$, is also one. 

Moreover, for the same reason, for any inert map $\alpha :\langle n \rangle \to \langle m \rangle$, the induced $\alpha_* :~\D^\otimes_{\langle n \rangle}\to~ \D^\otimes_{\langle m\rangle}$ sends $p^\otimes_{\langle n\rangle}$-Cartesian edges to $p^\otimes_{\langle m\rangle}$-Cartesian ones. 

Now, noting that Cartesian edges are a form of relative limit, the dual of \cite[4.3.1.15.]{HTT} implies that an edge in $\D^\otimes_{\langle n\rangle}$ is a $p^\otimes_{\langle n\rangle}$-Cartesian edge in $\D^\otimes_{\langle n\rangle}$ if and only if it is $p^\otimes$-Cartesian as an edge of $\D^\otimes$. 

In particular, $p^\otimes: \D^\otimes \to \E^\otimes$ admits Cartesian lifts for all morphisms lying over $id_{\langle n\rangle}$ (although it may not be a Cartesian fibration)

Therefore $\Fun_{\Fin_*}(\C^\otimes,\D^\otimes)\to \Fun_{\Fin_*}(\C^\otimes,\E^\otimes)$ is a Cartesian fibration. Indeed suppose you have a natural transformation $\eta$ of functors $\C^\otimes \to \E^\otimes$ lying over $\Fin_*$. Then each $\eta_c$ lies over $id_{\langle n\rangle}$, and so has a $p^\otimes$-Cartesian lift. It follows that $\eta$ itself has a Cartesian lift.

By base change, $\Fun^{lax/E}(\C^\otimes,\D^\otimes)\to \Fun^{lax}(\C^\otimes,\E^\otimes)$ is a Cartesian fibration, where $\Fun^{lax/E}(\C^\otimes,\D^\otimes)$ denotes the category of functors $\C^\otimes\to \D^\otimes$ over $\Fin_*$ whose projection to $\E^\otimes$ preserves inert edges. 

It therefore finally suffices to show that if $F\to G$ is Cartesian and $G$ preserves inert edges, then $F$ does so too (knowing that the projection of $F$ also preserves inert edges).

So let $f:x\to y$ be an inert edge in $\C^\otimes$, say it lies over $\alpha: \langle n \rangle \to \langle m \rangle$ in $\Fin_*$. Then $\alpha_*F(x)\to \alpha_*G(x)$ is $p^\otimes_{\langle m\rangle}$-Cartesian because $\alpha_*$ sends $p^\otimes_{\langle n \rangle}$-Cartesian edges to $p^\otimes_{\langle m \rangle}$-Cartesian ones; and so is $F(y)\to G(y)$. Moreover, since $G$ preserves inert edges, the canonical map $\alpha_*G(x)\to G(y)$ is an equivalence.

The following square commutes: 
\[
\xymatrix{\alpha_*F(x)\ar[r]\ar[d] & F(y) \ar[d]\\
\alpha_*G(x) \ar[r]^\simeq & G(y)}
\]

and both vertical arrows are $p^\otimes_{\langle m\rangle}$-cartesian. 

If we project this to $\E^\otimes$, we get a commutative square where both horizontal arrows are equivalences, it follows that the top horizontal arrows must also be an equivalence; which means precisely that $F(x)\to F(y)$ is inert and this is what we wanted to prove.  
\end{proof}
\begin{rem}
If we add the condition that $p$-cartesian edges are stable under tensor product, then one can prove that $\Fun^\otimes(\C,\D)\to \Fun^\otimes(\C,\E)$ is a Cartesian fibration; but the version here is all we need. 
\end{rem}
\bibliographystyle{alpha}
\bibliography{Biblio}
\end{document}